\documentclass[12pt,a4paper,reqno]{amsart}
\usepackage[maxbibnames=99,backend=biber, sorting=nyt, doi, url=false]{biblatex}
\usepackage{amsmath,amsthm,verbatim,amssymb,amsfonts,amscd, graphicx}
\usepackage{xcolor}
\usepackage{geometry}
\usepackage{graphicx}
\usepackage{hyperref}
\usepackage{enumitem}
\usepackage{cancel}
\usepackage[english,capitalize]{cleveref}
\definecolor{darkblue}{rgb}{0,0,0.3}
\definecolor{urlblue}{rgb}{0,0,0.7}
\hypersetup{
	colorlinks=true,
	citecolor=blue,
	linkcolor=darkblue,
	urlcolor=urlblue,
}
\usepackage{tikz}
\usepackage{pgfplots}
\usepgfplotslibrary{patchplots,colormaps}
\pgfplotsset{compat = newest}
\usepackage{mathrsfs} 
\usetikzlibrary{babel} 
\usetikzlibrary{cd} 

\newtheorem{lemma}{Lemma}
\newtheorem{definition}{Definition}
\newtheorem{exercise}{Exercise}
\newtheorem{question}{Question}
\newtheorem{theorem}{Theorem}
\newtheorem{corollary}{Corollary}
\theoremstyle{definition}
\newtheorem{remark}{Remark}

\renewcommand{\leq}{\leqslant}
\renewcommand{\geq}{\geqslant}

\renewcommand{\d}{\ensuremath{\mathrm d}}

\DeclareMathOperator{\Ric}{Ric}

\newcommand{\D}{\nabla}
\newcommand{\p}{\partial}

\DeclareMathOperator{\diam}{diam}
\renewcommand{\tilde}{\widetilde}
\renewcommand{\bar}{\overline}

\newcommand{\preg}{\p^{\text{reg}}}
\newcommand{\psing}{\p^{\text{sing}}}

\renewcommand{\epsilon}{\varepsilon}

\newcommand{\TODO}[1]{}

\title[Isoperimetric problems]{Isoperimetric problems and lower bounds on curvature}

\author{Gioacchino Antonelli}
\address{Gioacchino Antonelli
\hfill\break Department of Mathematics, University of Notre Dame, Hurley Hall, 255 Hurley, Notre Dame, IN 46556, United States}
\email{gantonel@nd.edu}

\begin{document}

\maketitle
\begin{abstract}
    This note surveys some classical results and recent developments on the interplay between lower curvature bounds and the isoperimetric problem. It is based on mini-courses given at the \emph{European Doctorate School of Differential Geometry} (Granada, July 2024) and the summer school \emph{Optimal transport, heat flow and synthetic Ricci bounds} (Chicago, June 2025).
\end{abstract}

\vspace{12pt}

\tableofcontents

\section{Introduction}
This note is intended as a gentle introduction to some aspects of the \textbf{interplay between lower curvature bounds and isoperimetry}. Many arguments are only sketched, simplified, or treated in particular cases. This is therefore not a paper with complete proofs, but rather a motivating guide for further study. The reader is encouraged to attempt the \textbf{Exercises} collected at the end, and to reflect on the \textbf{Open Questions} posed throughout the lectures, some of which touch on open problems in the field. For an overview of the material, we refer to the table of contents above and the four lectures into which this note is organized.
\smallskip

For the ease of readability most of the results and properties will be stated in the setting of Riemannian manifolds or limits of Riemannian manifolds, but they hold in the framework of RCD spaces as well. We refer the reader to the discussion before and after each statement, and to the papers referenced below.

\subsection{Further readings}
Here I list some useful general references related to the topics of this note. The reader interested in the isoperimetric problem (and how it interplays with lower bounds on the curvature) is encouraged to take a look at these references and the  references therein.
\begin{enumerate}
\item \emph{The isoperimetric problem} \cite{RosIsoperimetric05}. Survey paper by A. Ros.
    \item \emph{Isoperimetric inequalities in Riemannian manifolds} \cite{Ritore2023}. Book by M. Ritoré.
    \item \emph{Isoperimetric inequalities in unbounded convex bodies} \cite{LeonardiRitore}. Research paper by G.P. Leonardi, M. Ritoré, E. Vernadakis.
    \item \emph{Isoperimetry on manifolds with Ricci bounded below: overview of recent results and methods} \cite{PozzettaSurvey}. Survey paper by M. Pozzetta concerning part of the results I will discuss in this note.
    \smallskip
    \item \emph{Lectures on Optimal Transport} \cite{AmbrosioBrueSemolaBook}. A book in Optimal Transport. 
    \item \emph{Calculus, heat flow and curvature-dimension bounds in metric measure spaces} \cite{AmbrosioSurvey}. A survey on the early developments of the CD/RCD theory. 
    \item \textit{De Giorgi and Gromov working together} \cite{DeGiorgiGromov}. A survey on the developments of the CD/RCD theory and applications.
\end{enumerate}

\section{First Lecture}

\subsection{Preliminaries}
A general reference for results about sets of finite perimeter is Maggi's book \cite{MaggiBook}. In this note $(M^n,g)$ will be a complete smooth Riemannian manifold without boundary (if not otherwise stated). The Riemannian distance associated to $g$ will be denoted $\mathrm{d}$. The open (resp., closed) ball  of center $p\in M$ and radius $r>0$ will be denoted $B_r(p)$ (resp., $\bar B_r(p)$).

Let $k\geq 0$, $\delta>0$, and $E\subset M$. We denote the Hausdorff measure
\[
\mathcal{H}^k(E):=\sup_{\delta>0}\mathcal{H}^k_\delta(E):=\sup_{\delta>0} \inf\left\{2^{-k}\omega_k\sum_{i=1}^\infty\diam(A_i)^k: E\subset\bigcup_{i=1}^\infty A_i, \diam(A_i)\leq \delta \right\},
\]
where $\omega_k:=\pi^{k/2}\Gamma((k+2)/2)^{-1}$. When $k\in\mathbb N$, $\omega_k$ is the volume of the unit ball in $\mathbb R^k$. $\mathcal{H}^k(\cdot)$ is a Borel regular outer measure on $M^n$. 

Let $E\subset M^n$ be a Borel set. The \emph{volume} of $E\subset M^n$ is $\mathcal{H}^n(E)$. It will also be denoted  with $|E|$. Let $\alpha=0$ (resp., $\alpha=1$), and let $E\subset M$ be a Borel set. We denote $E^{(\alpha)}$ the \emph{essential exterior} (resp., the \emph{essential interior}) of $E$:
\[
E^{(\alpha)}:=\left\{x\in M: \lim_{r\to 0^+}\frac{|E\cap B_r(x)|}{|B_r(x)|}=\alpha\right\}.
\]
The \emph{essential boundary} of $E$ is $\partial^*E:=M\setminus (E^{(0)}\cup E^{(1)})$. 

The \emph{perimeter of $E$ inside an open set $\Omega\subset M$} is
\begin{equation}\label{eqn:PerimeterOnManifold}
P(E,\Omega):=\sup\left\{\int_E \mathrm{div}(X)\mathrm{d}\mathcal{H}^n: X\in C^\infty_{c}(\Omega;TM), g(X,X)\leq 1 \right\}.
\end{equation}

The \emph{perimeter} of $E$ is $P(E):=P(E,M)$. We say that $E$ has \emph{finite perimeter} if $P(E)<\infty$. 

\subsubsection{Properties and comments}\label{subsectPropAndComm}
\begin{enumerate}
    \item (Density $1/2$ and essential boundary) (De Giorgi--Federer) When $E$ has finite perimeter, it can be shown that $P(E,\Omega)=\mathcal{H}^{n-1}(\partial^* E\cap\Omega)$ for every open set $\Omega$. Moreover, for $\mathcal{H}^{n-1}$-almost every $x\in \partial^* E$ we have:
\[
\lim_{r\to 0^+}\frac{|E\cap B_r(x)|}{|B_r(x)|}=\frac{1}{2}.
\]
In addition, at $\mathcal{H}^{n-1}$-almost every $x\in \partial^*E$ \underline{the} (unique) blow-up of $\partial^*E$ at $x$ is a codimension-one Euclidean hyperplane in $T_xM$: a generalized unit normal $\nu(x)$ is defined at $\mathcal{H}^{n-1}$-almost every point $x\in \partial^*E$.
    \item (A general notion of perimeter) The notion of perimeter can be defined in a metric measure space $(X,\mathrm{d},\mathfrak{m})$, where $(X,\mathrm{d})$ is a complete separable metric space, and $\mathfrak{m}$ is a Radon measure which is finite on balls. See Ambrosio \cite{AmbrosioAhlfors}, and Miranda \cite{Miranda03}. In such a setting, given $f\in L^1(X,\mathfrak{m})$, we say that $f$ is of \emph{bounded variation}, or $f\in \mathrm{BV}(X,\mathrm{d},\mathfrak{m})$, if 
    \[
    |Df|(X):=\inf\liminf_{i\to\infty}\left\{\int_X|\nabla f_i|\mathrm{d}\mathfrak{m}:f_i\in\mathrm{Lip}_{\mathrm{loc}}(X,\mathrm{d}),f_i\rightarrow f\,\text{in $L^1_{\mathrm{loc}}(X,\mathfrak{m})$}\right\}<+\infty.
    \]
    Here $|\nabla f|(x):=\limsup_{y\to x}\frac{|f(y)-f(x)|}{\mathrm{d}(x,y)}$. An $\mathfrak{m}$-finite set $E\subset M$ has \emph{finite perimeter} when $\chi_E\in \mathrm{BV}(X,\mathrm{d},\mathfrak{m})$. This notion is consistent with the one given above \eqref{eqn:PerimeterOnManifold}.

    The notions of set of finite perimeter and function of bounded variation have been recently thoroughly studied in abstract metric measure spaces. In particular, the basic structure results known in $\mathbb R^n$ and Riemannian manifolds that we are discussing here in \cref{subsectPropAndComm} have been extended, as of today, to the setting of $\mathrm{RCD}(K,N)$-spaces with $K\in\mathbb R$ and $N\in [1,\infty)$ (see \cref{BMandRCD}), see, e.g., \cite{AmborsioBrueSemola19, BruePasqualettoSemola, BPSGaussGreen, AntonelliBrenaPasqualetto}.
    \item\label{precompactness} (Precompactness and lower semicontinuity) If $(E_i)\subset M$ is a sequence of sets of finite perimeter with 
    \[
    \sup_{i\in \mathbb N} P(E_i)<\infty,
    \]
    then up to subsequences 
    \[
    \chi_{E_i}\longrightarrow \chi_E, \qquad \text{in $L^1_{\mathrm{loc}}$},
    \]
    and 
    \[
    P(E)\leq \liminf_{i\to\infty} P(E_i).
    \]
    \item\label{Approx} (Approximation) Let $f\in\mathrm{BV}(M,\mathrm{d},\mathcal{H}^n)$. Then there exists $f_i\in C^\infty_c(M)$ such that $f_i\to f$ in $L^1$ and $\int_M|\nabla f_i|\to |Df|(M)$.
    \item\label{Coarea} (Coarea formula) For every $f\in \mathrm{Lip}(M)$ we have
    \[
    \int_M|\nabla f|\mathrm{d}\mathcal{H}^n=\int_{\mathbb R} P(\{f\geq t\})\mathrm{d}t.
    \]
    A general version of the co-area formula valid in arbitrary metric measure spaces can be found in \cite[Section 4]{Miranda03}.
\end{enumerate}

\subsection{Isoperimetric problem}

Let $(M^n,g)$ be a Riemannian manifold. Let $I_M:(0,|M|)\to [0,+\infty]$ be the \emph{isoperimetric profile} defined as:
\[
I_M(v):=\inf\left\{P(E):|E|=v\right\}.
\] 
If a Borel set $E\subset M^n$ satisfies $I_M(|E|)=P(E)$, then $E$ is called an \emph{isoperimetric set}. For discussions about the regularity of the function $I$ see the end of Lecture 1 and Lecture 2. When there is no risk of confusion we will write $I$ instead of $I_M$.
\begin{lemma}\label{lem:Existence}
    Let $(M^n,g)$ be a compact Riemannian manifold. Then for every $v\in (0,|M|)$ there exists an isoperimetric set with volume $v$.
\end{lemma}
\begin{proof}
    Take $(E_i)\subset M$ such that $|E_i|=v$ and $P(E_i)\to I(v)$ and apply \cref{precompactness} of the list of properties above.
\end{proof}
\begin{remark}[About existence in the non-compact case]
    \cref{lem:Existence} holds also when $M$ is possibly non-compact but has finite volume, due to Ritoré--Rosales \cite{RitRosales04}. When $M$ has infinite volume, one can come up with simple examples for which there are volumes $v$ such that no isoperimetric sets of volume $v$ exist (see \cref{ExNonExistence}). 
    Constructing examples without isoperimetric sets keeping the curvature bounded below is more challenging, see Lecture 4.
\end{remark}

{
The following result is due to Morgan, and leverages seminal contributions of Federer (\cite{FedererBullAMS}, \cite[Chapter 5]{Federer}), Almgren \cite{AlmgrenExistence}, and Bombieri \cite{BombieriRegularity}. For a modern treatment of the proof in the setting of volume-constrained perimeter minimizers in $\mathbb R^n$, which can be generalized to smooth Riemannian manifolds, the reader can consult \cite[Theorem 27.4 \& 28.1]{MaggiBook}.
\begin{theorem}[{\cite[Corollary 3.8]{morgan2003regularity}}]\label{thm:Regularity}
    Let $\tilde E\subset M^n$ be an isoperimetric set. Then there exists an open set $E\subset M^n$ with $|E\Delta\tilde E|=0$ such that $\partial E=\partial^* E$ and we can decompose:
    \[
    \partial E=\partial^{\mathrm{sing}}E\sqcup \partial^{\mathrm{reg}}E,
    \]
    where:
    \begin{enumerate}
        \item $\partial^{\mathrm{sing}}E$ is closed, and for every $d>n-8$\footnote{Actually: if $n\leq 7$, $\partial^{\mathrm{sing}}E=\emptyset$; and, if $n=8$, $\partial^{\mathrm{sing}}E$ has no accumulation points.}, $\mathcal{H}^d(\partial^{\mathrm{sing}}E)=0$;
        \item $\partial^{\mathrm{reg}}E$ is an open analytic hypersurface with constant mean curvature.
    \end{enumerate}
\end{theorem}
}

\subsubsection{Isoperimetric sets in $\mathbb R^n$}
We give a proof of the well-known isoperimetric inequality on $\mathbb R^n$.
\begin{theorem}
    Let $E\subset \mathbb R^n$ be a set of finite perimeter with $|E|=|B_1(0)|=\omega_n$. Then 
    \begin{equation}\label{eqn:Isoperimetric}
    P(E)\geq P(B_1(0)).
    \end{equation}
    If equality holds, then there is $p\in\mathbb R^n$ such that $|E\Delta B_1(p)|=0$.
\end{theorem}
\begin{proof}
    We only prove the inequality. For $A,B\subset \mathbb R$, let $A+B:=\{a+b:a\in A, b\in B\}$. Define the \emph{lower Minkowski content} of $E\subset \mathbb R^n$ as:
    \[
    \mathcal{M}(E):=\liminf_{r\to 0}\frac{|E+B_r|-|E|}{r}.
    \]
    It is known in general that the lower semicontinuous relaxation of $\mathcal{M}(\cdot)$ with respect to the convergence in measure (on finite measure sets) is the perimeter, see Ambrosio--Di Marino--Gigli \cite{ADMG17}. So it is enough to show that if $|E|=|B_1(0)|$, then 
    \[
    \mathcal{M}(E)\geq P(B_1(0)).
    \]
    By the Brunn--Minkowski inequality (\cref{BMinequality}) we have 
    \[
    \mathcal{M}(E) \geq \liminf_{r\to 0}\frac{\left(|E|^{1/n}+|B_r|^{1/n}\right)^n-|E|}{r}=n\omega_n=P(B_1(0)). 
    \]
    The proof by Steiner symmetrization is better suited to show the equality case: see \cite{MaggiBook}.
\end{proof}

\begin{remark}[Isoperimetric problem on model spaces and comments] We observe the following.
    \begin{enumerate}
        \item In $\mathbb R^n$ (but also in the round sphere $\mathbb S^n$, and the hyperbolic space $\mathbb H^n$) the only isoperimetric sets are geodesic balls. This is also true \emph{at the second order}. Namely, compact immersed oriented volume-preserving stable constant mean curvature (CMC) hypersurfaces in $\mathbb R^n$ (and also $\mathbb S^n,\mathbb H^n$) are geodesic spheres (\cite{BDC84, BDC88}).
        
        Moreover, in $\mathbb R^n$, compact oriented \emph{embedded} constant mean curvature hypersurfaces are geodesic spheres \cite{AlexandrovTheorem}. This is false if one substitutes embedded with immersed (see: Wente tori).
        \item There are several approaches to the proof of the isoperimetric inequality. The reader might consult, e.g., Lecture 4, the recent survey\footnote{\url{https://www.ams.org/journals/notices/202406/rnoti-p708.pdf}} by Brendle and Eichmair, or the references listed at the beginning of this note.
    \end{enumerate}
\end{remark}

\subsection{Brunn--Minkowski inequality and Curvature--Dimension condition}\label{BMandRCD}

{Let us give a proof of the Brunn--Minkowski inequality that is different from the one hinted at in \cref{BMinequality}. This proof, which leverages a result of \cite{McCann97}, unveils a connection with the theory of optimal transport and curvature-dimension conditions in the weak sense. {The statement of the Brunn--Minkowski inequality, which will be proved at the end of this section, is the following.
\begin{theorem}
    Let $A,B\subset \mathbb R^n$ be Borel sets with $|A|>0,|B|>0$. Then
    \[
|(1-t)A+tB|^{1/n} \geq (1-t)|A|^{1/n} + t|B|^{1/n}, \qquad \forall\, 0\leq t\leq 1.
\]
\end{theorem}
}

Let us recall that $e_i:X\times X\to  X$ denotes the projection on the $i$-th factor of the product, and $\mathbb P(X)$ denotes the set of probabilities on $X$.
For more details on what follows the interested reader can consult, e.g., Chapters 8, 10, 15, and 19 of \cite{AmbrosioBrueSemolaBook}. 
The reader can consult \cite[Chapter 10]{AmbrosioBrueSemolaBook} for other properties (e.g., being non-branching, or positively curved) of $(X,\mathrm{d})$ that are transmitted to $(\mathscr{P}_2(X),W_2)$.
\begin{definition}[Wasserstein space]
    Let $(X,\mathrm{d})$ be a complete and separable metric space, and define 
    \[
    \mathscr{P}_2(X):=\left\{\mu\in \mathbb P(X): \exists x_0\in X\,\text{s.t}\, \int_{X} \mathrm{d}(x,x_0)^2\mathrm{d}\mu(x)<+\infty\right\}.
    \]
    For $\mu,\nu\in \mathscr{P}_2(X)$, define
    \[
    W_2^2(\mu,\nu):=\min\left\{\int_{X\times X}\mathrm{d}^2(x,y)\mathrm{d}\pi(x,y):\pi\in \mathbb P(X\times X):(e_1)_\#\pi =\mu,(e_2)_\#\pi =\nu\right\}.
    \]
    Then $(\mathscr{P}_2(X),W_2)$ is a complete and separable metric space. Moreover, if $(X,d)$ is compact (resp., geodesic), $(\mathscr{P}_2(X),W_2)$ is also compact (resp., geodesic).
\end{definition}

\begin{definition}[Rény entropy]
Let $(X,\mathrm{d})$ be a complete and separable metric space. Let $\mathfrak{m}$ be a Borel measure which is finite on bounded sets. Let $\mu\in\mathscr{P}_2(X)$ and $N\geq 1$. Then, if we decompose $\mu=\rho\mathfrak{m}+\mu^s$, where $\rho\in L^1_{\mathrm{loc}}(\mathfrak{m})$ and $\mu^s\perp \mathfrak{m}$, the \emph{$N$-Rény entropy} of $\mu$ is 
\[
\mathcal{E}_N(\mu):=-\int_X \rho^{1-\frac{1}{N}}\mathrm{d}\mathfrak{m}.
\]
\end{definition}
An important feature of the $n$-Rény entropy is that it is convex along geodesics of $\mathscr{P}_2(\mathbb R^n)$. The next proposition, which implies Brunn--Minkowski inequality as we will see below, is due to McCann. We sketch its proof referring the reader to \cite{AmbrosioBrueSemolaBook} for the missing details.
\begin{theorem}[{\cite[Proposition 1.2]{McCann97}}]\label{McCann97Thm}
    The $n$-Rény entropy is convex along geodesics connecting two probability measures in $\mathscr{P}_2(\mathbb R^n)$ that are absolutely continuous with respect to Lebesgue $\mathscr{L}^n$.
\end{theorem}
\begin{proof}[Sketch]
    Let $\mathscr{P}_2(\mathbb R^n)\ni \mu_0,\mu_1\ll \mathscr{L}^n$. Let $\nabla\phi$, where $\phi:\mathbb R^n\to (\infty,\infty]$ is convex and lower semicontinuous, be the transport map from $\mu_0$ to $\mu_1$. {Namely,
$$
(\nabla\phi)_\# \mu_0 = \mu_1,
$$ 
and $\nabla\phi$ realizes the minimum of the Monge problem with quadratic cost:
$$
\int_{\mathbb R^n} |x - \nabla\phi(x)|^2 \, \mathrm{d}\mu_0(x)
=
\min_{T_\# \mu_0 = \mu_1}
\int_{\mathbb R^n} |x - T(x)|^2 \, \mathrm{d}\mu_0(x).
$$} Then the geodesic connecting $\mu_0$ to $\mu_1$ is
    \[
    \mu_t = ((1-t)\mathrm{id}+t\nabla\phi)_\# \mu_0=:(T_t)_\#\mu_0.
    \]
    Let us call $U(s):=-s^{1-\frac{1}{n}}$, and $\rho$ the density of $\mu_0$ with respect to $\mathscr{L}^n$. It can be shown that, for every $t\in [0,1]$,
    \[
    \mathcal{E}_n(\mu_t)=\int_{D_0\setminus\Sigma} U\left(\frac{\rho(x)}{\det\nabla T_t(x)}\right)\det\nabla T_t(x)\mathrm{d}\mathscr{L}^n,
    \]
    where $D_0$ is the $\mathscr{L}^n$-full-measure set of points $x$ where both $\phi$ and $\nabla\phi$ are differentiable, and $\Sigma:=\{x:\det\nabla T_t(x)=0\}$ is $\mathscr{L}^n$-negligible. Now observe that for every $x\in D_0\setminus\Sigma$ $a(t):=(\det\nabla T_t(x))^{1/n}$ is concave, and for every $z>0$, and $x\in D_0\setminus\Sigma$, we have that $b(t):=z^nU(\rho(x)z^{-n})$ is convex and non-increasing. Then $[0,1]\ni t \mapsto \mathcal{E}_n(\mu_t)$ is convex.
\end{proof}
As one can see inspecting the proof, the key property that we are using is that the function $U(s)$ defining the Rény entropy $\mathcal{E}_n$ has the following property: $(0,\infty)\ni s\mapsto s^n U(s^{-n})$ is convex and non-increasing.

The previous \cref{McCann97Thm} can also be understood as one of the starting point of the abstract theory of spaces with curvature lower bounded in the weak sense. 
\begin{definition}[Sturm \cite{Sturm1, Sturm2}, Lott--Villani \cite{LottVillani}]\label{defCD}
    Let $(X,\mathrm{d})$ be a complete and separable metric space. Let $\mathfrak{m}$ be a Borel measure which is finite on bounded sets. Let $N\geq 1$. We say that $(X,\mathrm{d},\mathfrak{m})$ \emph{satisfies the $\mathrm{CD}(0,N)$ condition} if the following holds. 

    For any $\mathscr{P}_2(X)\ni \mu_0,\mu_1\ll \mathfrak{m}$  with bounded support, there exists a geodesic $(\mu_t)_{t\in [0,1]}$ in $(\mathscr{P}_2(X),W_2)$ such that $\mu_t\ll \mathfrak{m}$ for every $t\in [0,1]$ and $[0,1]\ni t\mapsto \mathcal{E}_{N'}(\mu_t)$ is convex, for all $N'\geq N$.
\end{definition}

\begin{proof}[Proof of Brunn--Minkowski inequality using \cref{McCann97Thm}]
Let $A,B$ be Borel sets with $\mathscr{L}^n(A)>0$ and $\mathscr{L}^n(B)>0$. Let $\mu_t$ be the $W_2$-geodesic connecting $\mu_A:=\frac{\chi_A}{\mathscr{L}^n(A)}\mathscr{L}^n$ and $\mu_B:=\frac{\chi_B}{\mathscr{L}^n(B)}\mathscr{L}^n$. Thus, by \cref{McCann97Thm} we have, for every $t\in [0,1]$,
\[
-\mathcal{E}_n(\mu_t) \geq -(1-t)\mathcal{E}_n(\mu_A) - t\mathcal{E}_n(\mu_B) = (1-t)\mathscr{L}^n(A)^{1/n} + t\mathscr{L}^n(B)^{1/n}.
\]
Notice that by Jensen inequality, whenever a probability measure $\mu=\rho\mathscr{L}^n$ is supported on a closed set $C$, then $-\mathcal{E}_n(\mu)\leq \mathscr{L}^n(C)^{1/n}$. Thus, since $\mu_t$ is supported on the set $(1-t)A+tB$ (the optimal transport takes place along geodesics, see \cite{AmbrosioBrueSemolaBook} for more details), we finally get
\[
\mathscr{L}^n((1-t)A+tB)^{1/n} \geq (1-t)\mathscr{L}^n(A)^{1/n} + t\mathscr{L}^n(B)^{1/n},
\]
as desired.
\end{proof}

Verbatim as in the proof of Brunn--Minkowski inequality above we have the following. \cref{BMonCD0n} will be useful in Lecture 4.
\begin{theorem}[Brunn--Minkowski inequality on $\mathrm{CD}(0,N)$ spaces]\label{BMonCD0n}
    Let $N\geq 1$, and  $(X,\mathrm{d},\mathfrak{m})$ be a $\mathrm{CD}(0,N)$ metric measure space. Let $A_0,A_1$ be Borel sets with positive $\mathfrak{m}$ measure. For every $t\in [0,1]$ let us define 
    \[
    A_t:=\{y\in X: \exists (x_0,x_1)\in A_0\times A_1,\,\, \mathrm{d}(x_0,y)=t\mathrm{d}(x_0,x_1),\,\, \mathrm{d}(y,x_1)=(1-t)\mathrm{d}(x_0,x_1) \}.
    \]
    Then, for every $t\in [0,1]$ and $N'\geq N$ we have 
    \[
    \mathfrak{m}(A_t)^{1/N'} \geq (1-t)\mathfrak{m}(A_0)^{1/N'} + t\mathfrak{m}(A_1)^{1/N'}.
    \]
\end{theorem}

We do not delve into giving the precise definition of $\mathrm{CD}(K,N)$ and $\mathrm{RCD}(K,N)$ spaces when $K$ might be non-zero, or $N$ might be infinite. We refer the reader to the seminal papers \cite{LottVillani, Sturm1, Sturm2, AGS1Inve, AGSDuke, Gigli12}. We only remark that on a complete $N$-dimensional Riemannian manifold the $(\mathrm{R})\mathrm{CD}(K,N)$-condition is equivalent to $\mathrm{Ric}\geq K$, see \cite[Theorem 1.7]{Sturm2}.

}

\subsection{Statement of the sharp concavity inequality}
In Lecture 2 we are going to prove the following result. We start by stating it. We focus on the case where $M$ is a Riemannian manifold, but the proof works for arbitrary $n$-dimensional $\mathrm{RCD}(k,n)$-spaces $(X,\mathrm{d},\mathcal{H}^n)$, see \cite{ConcavitySharpAPPS}.

\begin{theorem}[Bavard--Pansu \cite{BavardPansu86}, Bray \cite{BrayPhD}, Bayle \cite{Bayle04}, A--Pa--Po--S \cite{ConcavitySharpAPPS}]\label{thm:ConcavitySharp}
    Let $(M^n,g)$ be a complete Riemannian manifold with $\inf_{x\in M}|B_1(x)|>0$. Assume that $\mathrm{Ric}\geq k\cdot g$ as quadratic forms, for some $k\in\mathbb R$. Then
    \begin{equation}\label{eqn:Viscosity}
    -I_M''I_M \geq k + \frac{(I_M')^2}{n-1}, \qquad \textit{holds in the viscosity sense on $(0,|M|)$},
    \end{equation}
    {i.e., for every $v\in (0,|M|)$ and every smooth real-valued function $\varphi$ such that $\varphi(v)=I_M(v)$ and $\varphi(x)\leq I_M(x)$ in a neighborhood of $v$, we have
    \[
    -\varphi''(v)\varphi(v)\geq k+\frac{\varphi'(v)^2}{n-1}.
    \]
    }
\end{theorem}

\begin{remark}[About $I_M>0$]\label{rem:I>0}
    Notice that if $M$ is compact, the hypothesis $\inf_{x\in M}|B_1(x)|>0$ is redundant. 

    Also, if $M$ is compact, isoperimetric sets with volume $v$ exist for every $v\in (0,|M|)$, and thus (since they have positive perimeter by the local isoperimetric inequality) $I(v)>0$ for every $v\in (0,|M|)$. When $M$ is non-compact, one can still prove that $\inf_{x\in M}|B_1(x)|>0$ and $\mathrm{Ric}\geq k\cdot g$ imply that $I(v)>0$ for every $v\in (0,|M|)$: see, e.g., \cite{AFP21}.
\end{remark}
\begin{remark}[Regularity of $I_M$]\label{rem:RegularityIM}
    Under the assumptions of \cref{thm:ConcavitySharp} one can directly prove\footnote{The classical argument is due to Gallot \cite{Gallot88}, see also Munoz Flores--Nardulli \cite{FloresNardulli20}.} that $I_M$ is locally $\frac{n-1}{n}$-H\"older continuous on $(0,|M|)$, thus continuous. 
    
    The inequality \eqref{eqn:Viscosity} holds also in the appropriate distributional sense, after rewriting it as a differential inequality on $I_M^{\frac{n}{n-1}}$. Classical analysis (see \cref{ExRegularity}) tells us that \eqref{eqn:Viscosity}, joined with the fact that $I_M$ is continuous, implies that it is locally semi-concave. Hence $I_M'(v)$ exists for co-countably many $v\in (0,|M|)$, $I_M''(v)$ exists for almost every $v\in (0,|M|)$, and $I_M''$ is a locally bounded measure. Moreover, as a consequence, \eqref{eqn:Viscosity} also holds pointwise almost everywhere.
    \end{remark}
    \begin{remark}[The case $k=0$]
When $k=0$ the inequality \eqref{eqn:Viscosity} is equivalent to $(I^{\frac{n}{n-1}})''\leq 0$ in the distributional sense. Hence, since $I$ is continuous, the latter is equivalent to saying $I^{\frac{n}{n-1}}$ is concave.
\end{remark}
    \begin{question}\label{Questionk-1}
        Assume $k=-1$. Is \eqref{eqn:Viscosity} still true even if $\inf_{x\in M}|B_1(x)|=0$?
    \end{question}
    
    Understanding the precise regularity, or even better the precise expression of $I_M$ might be a very challenging problem, already in some fairly simple explicit cases, see the questions below.

\begin{question}\label{QuestionDifferentiability}
    Construct, if possible, a complete non-compact $(M^n,g)$ with $\mathrm{Ric}\geq 0$, $\inf_{x\in M}|B_1(x)|>0$, and such that the following holds: there exists a sequence of $v_i\to\infty$ such that $I_M$ is not differentiable at $v_i$.
\end{question}
\begin{question}
    Let $M=[0,1]^3$. Compute $I_M$. Here, we consider the \emph{relative isoperimetric problem}: the perimeter only takes into account the part of $E$ inside $M$, i.e., in $(0,1)^3$. See \cite{RitoreIsopCube} and related references for contributions on the latter problem. See also the recent \cite{MilmanCubesSlabs}.
\end{question}

\newpage

\section{Second Lecture}

\subsection{Proof of the sharp concavity inequality}
\begin{proof}[Proof of \cref{thm:ConcavitySharp}]
    There are three increasing levels of difficulties. We will go through showing that \eqref{eqn:Viscosity} holds in the barrier sense. {By that we mean that for every $v\in (0,|M|)$, and every $\varepsilon>0$, there is a smooth function $A:(v-\delta,v+\delta)\to \mathbb R$ defined on some neighborhood of $v$ such that $I_M(v)=A(v)$, $I_M(x)\leq A(x)$ for every $x\in (v-\delta,v+\delta)$, and 
    \[
    -A''(v)A(v)\geq k+\frac{A'(v)^2}{n-1}-\varepsilon.
    \]}
    We leave to the reader as an exercise to show that if \eqref{eqn:Viscosity} holds in the barrier sense, then it holds in the viscosity sense. 
    \medskip

    \textbf{Case 1.} $M^n$ is compact and $n\leq 7$. The barrier inequality follows from computing the second variation of the area of an isoperimetric set $E$. Let $E$ be an isoperimetric set with volume $v_0\in (0,|M|)$. It exists, and $\partial E$ is smooth by \cref{lem:Existence}, and \cref{thm:Regularity}, respectively.
    
    For $\varphi\in C^\infty(M)$, let us consider a smooth family of sets $\{E_t\}_{t\in(-\varepsilon,\varepsilon)}$, such that $E_0=E$, the variational vector field $X_t$ along $\p E_t$ at $t=0$ is $\varphi\nu$ (where $\nu$ denotes the outer unit normal of $\p E_t$), and $\nabla_{X_t}X_t = 0$ at $t=0$. For example $\partial E_t:=F_t(\partial E)$, where $F_t(x):=\exp(t\varphi(x)\nu(x))$ for $x\in\partial E$, satisfies the assumptions.
    
    Define $V(t):=|E_t|$ and $A(t):=P(E_t)$; thus $V(0)=v_0$, $A(0)=I(v_0)$. Let $H$ be the mean curvature of $\partial E$. We compute the first and second variations of the volume at the initial time:
    \begin{equation}\label{eqn:VariationsVolume}
        \frac{\d V}{\d t}(0)=\int_{\p E} \varphi,\qquad \frac{\d ^2V}{\d t^2}(0)=\int_{\p E}H\varphi^2. 
    \end{equation}
    Next, we compute the first variation of the area:
    \begin{equation}\label{eqn:FirstVariationOfArea}
        \frac {\d A}{\d t}(0)=\int_{\p E} H\varphi,
    \end{equation}
    and, finally, the second variation of the area is
    \begin{equation}\label{eqn:SecondVariation}
        \begin{aligned}
            \frac {\d^2 A}{\d t^2}(0) &= \int_{\p E} \Big(-\Delta_{\p E}\varphi-\Ric(\nu,\nu)\varphi-|\mathrm{II}|^2\varphi\Big)\varphi + H^2\varphi^2. \\
        \end{aligned}
    \end{equation}
    
    Since $E$ is a volume-constrained minimizer, we must have $\frac{\mathrm{d}A}{\mathrm{d}t}(0)=0$ whenever $\frac{\mathrm{d}V}{\mathrm{d}t}(0)=0$. Hence $H$ is constant due to \eqref{eqn:FirstVariationOfArea} and the first of \eqref{eqn:VariationsVolume}.\footnote{We have just re-discovered the fact that boundaries of isoperimetric sets have constant mean curvature.} From now on, we fix the choice $\varphi=1$ in the variation. Since $V(t)$ is strictly monotone in $t$ in a neighbourhood of 0, due to the first of \eqref{eqn:VariationsVolume}, we may view $A$ as a smooth function in $V$, defined in some neighborhood of $v_0$. The value of the constant $H$ is thus obtained through the chain rule:
    \begin{equation}\label{eqn:DefineA'}
        A'(v_0)=\frac{\d A}{\d V}(v_0)=\frac{\frac{\d A}{\d t}(0)}{\frac{\d V}{\d t}(0)}=\frac{\int_{\p E} H}{\int_{\p E} 1}=H.\ 
    \end{equation}
    Applying our choice $\varphi=1$ to \eqref{eqn:SecondVariation}, and the trace inequality $|\mathrm{II}|^2\geq H^2/(n-1)$, we obtain
    \begin{equation}\label{eqn:EstimateSecondVariation}
        \frac {\d^2 A}{\d t^2}(0) = \int_{\p E} -\mathrm{Ric}(\nu,\nu)-|\mathrm{II}|^2+H^2 \leq \int_{\p E} -\mathrm{Ric}(\nu,\nu)+\frac{n-2}{n-1}H^2.
    \end{equation}
    Next, by using the chain rule and the formulae for the derivatives of the inverse function, \begin{equation}\label{eqn:ChainRule}
        A''(v_0)= \left(\frac{\d V}{\d t}(0)\right)^{-2}\frac{\d^2 A}{\d t^2}(0) - \left(\frac{\d V}{\d t}(0)\right)^{-3}\frac{\d A}{\d t}(0)\cdot \frac{\d^2 V}{\d t^2}(0).
    \end{equation}
    Thus, joining \eqref{eqn:ChainRule}, \eqref{eqn:EstimateSecondVariation}, \eqref{eqn:VariationsVolume}, and \eqref{eqn:FirstVariationOfArea}, and using $\mathrm{Ric}(\nu,\nu)\geq k$,  we get
    \begin{align}
        P(E)^2\cdot A''(v_0) &\leq \int_{\p E} \left(-\mathrm{Ric}(\nu,\nu)+\frac{n-2}{n-1}H^2\right) - \int_{\p E} H^2\\
        &\leq \left(-k-\frac{1}{n-1}H^2\right)P(E).
    \end{align}
    By using the previous inequality, the fact that $H=A'(v_0)$, and $P(E)=A(v_0)$, we finally obtain
    \begin{equation}\label{eqn:BarrierEquation}
        A(v_0)A''(v_0)\leq-\frac{A'(v_0)^2}{n-1}-k.
    \end{equation}
    Notice that $I(v_0)=A(v_0)$ and $I(v)\leq A(v)$ in some neighborhood of $v_0$, so $A$ is an upper barrier of $I$ which satisfies \eqref{eqn:BarrierEquation}. Thus \eqref{eqn:Viscosity} holds in the barrier sense, as stated at the beginning of this proof, and thus in the viscosity sense. This proves the desired result.
    \medskip

    \textbf{Case 2.} $M^n$ is compact and $n\geq 8$. We have to be careful about the variation here, because $\partial E$ might have a singular part $\partial^{\mathrm{sing}}E$. Let $E\subset M$ be an isoperimetric set with volume $v\in (0,|M|)$. It exists (\cref{lem:Existence}), and, by \cref{thm:Regularity} it has a representative whose boundary $\partial E$ can be decomposed as $\partial^{\mathrm{reg}}E\sqcup \partial^{\mathrm{sing}}E$, where $\mathcal{H}^d(\partial^{\mathrm{sing}}E)=0$ for every $d>n-8$, and $\partial^{\mathrm{reg}}E$ is a smooth open hypersurface with constant mean curvature.

    Hence, for each $\delta\ll 1$, we can find a finite collection of balls $B(x_i,r_i)$ covering $\psing E$ with $x_i\in \psing E$ and $r_i<\delta$, such that $\sum r_i^{n-7}\leq1$. For each $i$, we find a smooth function $\eta_i$ such that
    \[\eta_i|_{B(x_i,2r_i)}=0,\quad \eta_i|_{M\setminus B(x_i,3r_i)}=1,\quad |\D_M\eta_i|\leq 2r_i^{-1}.\]
    
    We claim that, by using the minimality property, for each $x\in M$ and $r<1$ we have
    \begin{equation}\label{eq:area_in_ball}
        P(E,B_r(x))\leq Cr^{n-1},
    \end{equation}
    where $C$ depends only on $M$. The constant $C$ might change from line to line from now on. To see \eqref{eq:area_in_ball}, for each $x\in M$ and $r>0$ there is a radius $s\in[0,r]$ such that $|B_r(x)\setminus B_s(x)|=|B_r(x)\cap E|$. This implies that the set
    \[
    E'=(E\cup B_r(x))\setminus B_s(x)
    \]
    satisfies $|E'|=|E|$. On the other hand, we have
    \[P(E')\leq P(E,\overline{B}_r(x)^c)+P(B_r(x))+P(B_s(x))\leq P(E,\overline{B}_r(x)^c) + Cr^{n-1},
    \]
    and
    \[
    P(E')\geq P(E)\geq P(E,\overline{B}_r(x)^c)+P(E,B_r(x)).
    \]
    This proves \eqref{eq:area_in_ball}. By regularizing $\overline\eta:=\min_i\{\eta_i\}$, we can find a function $\eta\in C^\infty(M)$ such that
    \[\eta=0\ \text{on}\ \bigcup B(x_i,r_i),\qquad
        \eta=1\ \text{on}\ M\setminus\bigcup B(x_i,4r_i),\]
    and $|\D_M\eta|\leq2|\D_M\bar\eta|$. Combined with \eqref{eq:area_in_ball}, and $|\D_M\eta_i|\leq Cr_i^{-1}$, we obtain
    \begin{equation}\label{eq:total_grad}
        \begin{aligned}
            \int_{\preg E}|\D_{\preg E}\eta|^2 &\leq \int_{\preg E}|\D_M\eta|^2\leq2\sum_i\int_{\preg E\cap B(x_i,4r_i)}|\D_M\eta_i|^2 \\
            &\leq C\sum_i r_i^{n-1}\cdot r_i^{-2}\leq C\delta^4.
        \end{aligned}
    \end{equation}

    By taking the exponential normal variation, for $\varphi\in C^\infty(M)$, let us consider a smooth family of sets $\{E_t\}_{t\in(-\varepsilon,\varepsilon)}$, such that $E_0=E$, the variational vector field $X_t$ along $\p E_t$ at $t=0$ is $\varphi\eta\nu$ (where $\nu$ denotes the outer unit normal of $\p E_t$), and $\nabla_{X_t}X_t = 0$ at $t=0$. This family is well-defined since $\eta$ is supported inside $\preg E$. Doing the computations as in \textbf{Step 1}, and fixing $\varphi=1$, we get that $H=A'(v_0)$ is constant on $\preg E$, and
    \begin{equation}\label{eqn:VariationsVolumenew}
        \frac{\d V}{\d t}(0)=\int_{\preg E} \eta,\qquad \frac{\d ^2V}{\d t^2}(0)=\int_{\preg E}H\eta^2,
    \end{equation}\begin{equation}\label{eqn:FirstVariationOfAreanew}
        \frac {\d A}{\d t}(0)=\int_{\preg E} H\eta,
    \end{equation}
    \begin{equation}\label{eqn:SecondVariationnew}
        \begin{aligned}
            \frac {\d^2 A}{\d t^2}(0) &\leq \int_{\preg E} |\nabla_{\preg}\eta|^2 +\eta^2\left(-k+\frac{n-2}{n-1}H^2\right). \\
        \end{aligned}
    \end{equation}
    Thus, using \eqref{eqn:ChainRule}, arguing as in \textbf{Case 1}, calling $Q:=\int_{\preg E}\eta$, and finally using \eqref{eq:total_grad}, we get 
    \begin{equation}\label{eqn:LAFINE}
    \begin{aligned}
    Q^2\cdot A''(v_0) &\leq \int_{\preg E} |\nabla_{\preg}\eta|^2 +\eta^2\left(-k-\frac{1}{n-1}H^2\right)\\
    &\leq C\delta^4 + \left(-k-\frac{1}{n-1}A'(v_0)^2\right)\int_{\preg E}\eta^2
    \end{aligned}
    \end{equation}
    Now notice that $Q^{-2}\int_{\preg E}\eta^2 \geq P(E)^{-1}=A(v_0)^{-1}$ by H\"older inequality. Moreover, possibly taking a smaller $\delta$, we can choose $\eta$ such that $Q\geq \tilde C$, where $\tilde C$ possibly depends on $E$. Then \eqref{eqn:LAFINE} implies that (notice that the function $A$ also depends on $\delta$)
    \[
    A''(v_0)\leq o_{\delta\to 0}(1)+ \left(-k-\frac{1}{n-1}A'(v_0)^2\right)A(v_0)^{-1}.
    \]
{Thus we have that, for every $\varepsilon>0$, there exists an upper-touching barrier $A$ at $v_0$ such that
\[
-A''A \geq k+\frac{A^2}{n-1}-\varepsilon,
\]
holds at $v_0$, i.e., \eqref{eqn:Viscosity} holds in the barrier sense, as desired.}
\medskip

    \textbf{Case 3.} $M^n$ is non-compact. For simplicity, we will only deal with the case $k=0$. The detailed proof is a bit technical and involves the use of recent techniques from non-smooth analysis. We will give a fairly complete sketch, by using the result in \cref{MeanCurvatureMonotonicity}. The proof of such a result in an even more general setting is in \cite{ConcavitySharpAPPS}. Let us write down the idea, and postpone the proof.

    \textbf{Idea}. We have to overcome this situation: for some volume $v$ there are no isoperimetric sets with volume $v$ (see Lecture 3). In this case we can prove, see \cref{concentrationcompactness}, that there is a Ricci-limit space at infinity where one has an isoperimetric set with volume $v$. Then, we will perform the previous computations in a weak sense on this isoperimetric set to recover the result. We need to develop some other tools before giving a full proof, so we postpone it to Lecture 3.
\end{proof}
{
\begin{question}
    Is there a way to obtain the proof of \textbf{Case 3.} of \cref{thm:ConcavitySharp} without leveraging the theory of $\mathrm{RCD}$ spaces?
\end{question}
}
\begin{remark}[Equality in \eqref{eqn:Viscosity}]\label{rem:Equality}
    If $M$ is isometric to the simply connected model of dimension $n$, with constant sectional curvature such that $\mathrm{Ric} = k\cdot g$, then $I_M$ satisfies the equality in \eqref{eqn:Viscosity}. Let us outline a way to verify the previous assertion when $M$ is the round sphere of constant sectional curvature $=1$, and of dimension $n$ (hence $k=(n-1)$).

    The round sphere of dimension $n$ and constant sectional curvature $=1$ is isometric to the warped product
    \[
    ((0,\pi)\times \mathbb S^{n-1},\mathrm{d}r^2+\sin^2(r)g_{\mathbb {S}^{n-1}}),
    \]
    where $g_{\mathbb S^{n-1}}$ is the metric of constant sectional curvature $=1$ on $\mathbb S^{n-1}$. We stated before the (only) isoperimetric sets are (equivalent to) geodesic spheres. Hence $I_{\mathbb S^n}$ is implicitly defined on $(0,|\mathbb S^n|)$ by
    \begin{equation}\label{eqn:ImplicitISn}
    I_{\mathbb S^n}\left(|\mathbb S^{n-1}|\int_0^r \sin^{n-1}(s)\mathrm{d}s\right) = |\mathbb S^{n-1}|\sin^{n-1}(r), \qquad \forall r\in (0,\pi).
    \end{equation}
    Differentiating \eqref{eqn:ImplicitISn} we get \eqref{eqn:Viscosity} with $k=n-1$.

    {In general, in the non-compact case, equality in \eqref{eqn:Viscosity} does not force any rigidity on the space. For examples, for the manifolds constructed in \cite[Theorem 1.1]{AntonelliGlaudo}, $(I^{\frac{n}{n-1}})''(v)=0$ for every $v\in (0,1)$, thus reaching equality in \eqref{eqn:Viscosity} for $k=0$. Nevertheless, those examples are not flat, and not even cones. Anyway, rigidity sometimes holds in the non-compact case under additional assumptions, compare with \cite[Theorem 1.2(3)]{ConcavitySharpAPPS2}. 
    
    On the other hand, e.g., if on a compact manifold $M^n$ with $2\leq n\leq 7$ equality holds at $v\in (0,|M|)$ in \eqref{eqn:Viscosity}, then every isoperimetric set $E$ of volume $v$ has totally umbilical boundary, and $\mathrm{Ric}(\nu_{\partial E},\nu_{\partial E})=k$, where $\nu_{\partial E}$ is the outward unit normal to $E$: compare with \cite[Theorem 3.13]{Ritore2023}.}
\end{remark}

\subsection{Bishop--Gromov and L\'evy--Gromov inequalities}
We prove the following. The results below in \cref{LGBG} are well-known and due to, respectively, Gromov and Bishop. Up to our knowledge, the proof of \cref{LG} in \cref{LGBG} we propose below is due to Bayle \cite{Bayle04}, and the proof of \cref{BG} we propose below in \cref{LGBG} is due to Bray \cite{BrayPhD}.
\begin{corollary}\label{LGBG}
    Let $(M^n,g)$ be a Riemannian manifold such that $\mathrm{Ric}\geq (n-1)\cdot g$ as quadratic forms. Then the following hold:
    \begin{enumerate}
        \item\label{LG} (L\'evy--Gromov isoperimetric inequality). We have
        \begin{equation}\label{eqn:LG}
        \frac{I_M(t|M|)}{|M|}\geq \frac{I_{\mathbb S^n}(t|\mathbb S^n|)}{|\mathbb S^n|}, \qquad \forall t\in [0,1].
        \end{equation}
        If equality holds in \eqref{eqn:LG} at some $t\in (0,1)$, then $M$ is isometric to $\mathbb S^n$.
        \item\label{BG} (Bishop--Gromov inequality) We have 
        \[
        |M|\leq |\mathbb S^n|.
        \]
        If equality holds, then $M$ is isometric to $\mathbb S^n$.
    \end{enumerate}
\end{corollary}
\begin{remark}[Bishop--Gromov monotonicity formula]
    \cref{BG} in \cref{LGBG} is a special case of a general monotonicity formula, due to Bishop--Gromov. We recall it here.
    
    Let $k\in\mathbb R$ and assume $(M^n,g)$ is a Riemannian manifold with $\mathrm{Ric}\geq k\cdot g$. Denote $V_{n,k}(r)$ the volume of any geodesic ball of radius $r$ in the simply connected space of dimension $n$, constant sectional curvature, and with $\mathrm{Ric}=k\cdot g$. Then, for every $p\in M$, 
    \begin{equation}\label{eqn:BishopGromov}
    \frac{|B_R(p)|}{V_{n,k}(R)}\leq \frac{|B_r(p)|}{V_{n,k}(r)}\leq 1 \qquad \forall 0<r<R.
    \end{equation}
    {Actually \eqref{eqn:BishopGromov} is just a particular case of a more general inequality: for every $0\leq r_1<r_2\leq r_3<r_4$ and every $p\in M$ we have,
    \begin{equation}\label{eqn:BG+}
    \frac{|B_{r_4}(p)|-|B_{r_3}(p)|}{V_{n,k}(r_4)-V_{n,k}(r_3)} \leq \frac{|B_{r_2}(p)|-|B_{r_1}(p)|}{V_{n,k}(r_2)-V_{n,k}(r_1)}.
    \end{equation}
    As a result of a limiting procedure involving \eqref{eqn:BG+}, the same monotonicity formula as in \eqref{eqn:BishopGromov} holds if one considers the perimeter of balls instead of the volume. 
    }

\end{remark}
\begin{proof}[Proof of \cref{LGBG}]
    Let us prove the two items independently. Notice that $\mathrm{Ric}\geq (n-1)\cdot g$ implies that $M$ is compact by the classical Bonnet--Myers theorem.
    \begin{enumerate}
        \item We only prove the inequality. Let us call $f(t):=\left(\frac{I_M(t|M|)}{|M|}\right)^{\frac{n}{n-1}}$, and $g(t):=\left(\frac{I_{\mathbb S^n}(t|\mathbb S^n|)}{|\mathbb S^n|}\right)^{\frac{n}{n-1}}$. By \cref{rem:Equality} and \eqref{eqn:Viscosity} we have 
        \[
        f''\leq -nf^{(2-n)/n}\,\, \text{in the viscosity sense on $(0,1)$}, \qquad g''=-ng^{(2-n)/n}\,\,\text{on $(0,1)$}.
        \]
        Hence ODE comparison (see \cref{ExODEDetails}) will give the desired inequality $f\geq g$. The equality case comes from a more refined analysis that we leave to the interested reader.
        \item We only prove the inequality. Set $\psi:=I^{\frac{n}{n-1}}$. A direct computation gives that, from \eqref{eqn:Viscosity}, we infer
    \begin{equation}\label{eqn:PsiZetaViscosity}
    \psi'' \leq -n\psi^{(2-n)/n} \qquad \text{in the viscosity sense on $(0,|M|)$.}
    \end{equation}
    For every $\zeta>0$, define $I_\zeta:(0,V_\zeta)\to\mathbb R$ implicitly by ($0<r<\pi$)
    \begin{equation}\label{eqn:DefnIzetaNew}
        I_\zeta\Big(\zeta\int_0^r\sin(s)^{n-1}\,\d s\Big)=\zeta\sin(r)^{n-1}.
    \end{equation}
    Notice $V_\zeta=\zeta\frac{|\mathbb S^n|}{|\mathbb S^{n-1}|}$, and $I_\zeta(0)=I_\zeta(V_\zeta)=0$. A direct computation shows that $\psi_\zeta:=I_\zeta^{\frac{n}{n-1}}$ satisfies 
    \begin{equation}\label{eqn:PsiZeta}
    \psi_\zeta'' = -n\psi_\zeta^{(2-n)/n} \qquad \text{on $(0,V_\zeta)$.}
    \end{equation}
    Notice that for small $v>0$, the function $I(v)$ is bounded above by the perimeter of a ball of volume $v$ centered at a point $x\in M$. Thus, denoting by $\psi'_+(0)$ the upper right-hand Dini derivative, we obtain
    \[
    \psi'_+(0)\leq n^{\frac{n}{n-1}}|\mathbb B^n|^{\frac1{n-1}},
    \]
    where $\mathbb B^n$ is the Euclidean unit ball in $\mathbb R^n$.
    Moreover, a direct computation gives 
    \[
    (\psi_\zeta)'_+(0)=n\zeta^{\frac1{n-1}}.
    \]
    Notice that when $\zeta=n|\mathbb B^n|=|\mathbb S^{n-1}|$, we have $V_{\zeta}=|\mathbb S^n|$.
    
    Suppose by contradiction that $|M|>|\mathbb S^n|$. Then there is $\zeta'>n|\mathbb B^n|$ such that $V_{\zeta'}<|M|$. Since $\zeta'>n|\mathbb B^n|$ we have 
    \begin{equation}\label{eqn:ConditionOnDerivative}
    +\infty>(\psi_{\zeta'})'_+(0) > \psi'_+(0).
    \end{equation}
    Taking into account \eqref{eqn:PsiZetaViscosity}, \eqref{eqn:PsiZeta}, and \eqref{eqn:ConditionOnDerivative} we can apply ODE comparison (see \cref{ExODEDetails}), and hence we get
    \[
    \psi(v)\leq \psi_{\zeta'}(v) \qquad \text{for all $v \leq V_{\zeta'}$}.
    \]
    The latter inequality, evaluated at $v=V_{\zeta'}$, gives $\psi(V_{\zeta'})\leq 0$, which is a contradiction with the fact that $I>0$ (see \cref{rem:I>0}) on $(0,V_{\zeta'})\subset (0,|M|)$. \qedhere
    \end{enumerate}
\end{proof}

{
\subsubsection{Localization technique and Lévy--Gromov isoperimetric inequality}\label{Localization}

Recently, in \cite{CM17}, the authors extended the Lévy--Gromov isoperimetric inequality to the setting of (essentially non-branching) $\mathrm{CD}(K,N)$ spaces, with $K\in\mathbb R$ and $1\leq N<\infty$. This has been done using the localization technique first worked out in the smooth weighted Riemannian setting by Klartag in \cite{Klartag}, and coming from ideas in convex geometry (see, e.g., \cite{PayneWeinberger}). Let us quickly describe the idea behind the proof. The localization technique will also play a key role in showing \eqref{eqn:Distributional} in \cref{MeanCurvatureMonotonicity}, which is instrumental in proving \cref{thm:ConcavitySharp} in the non-compact case. For what follows, the interested reader might consult \cite[Theorem 1.5]{Klartag}, and \cite[Theorem 5.1]{CM17}.
\medskip

Let us consider an essentialy non-branching $\mathrm{CD}(K,N)$ metric measure space $(X,\mathrm{d},\mathfrak{m})$ for some $K\in\mathbb R$ and $1\leq N<\infty$. Let $f:X\to\mathbb R$ be an $\mathfrak{m}$-integrable function with $\int_X f\mathrm{d}\mathfrak{m}=0$, and such that there is $x_0$ for which $\int_X f(x)\mathrm{d}(x,x_0)\mathrm{d}\mathfrak{m}<\infty$.
The function $f$ induces an $\mathfrak{m}$-almost everywhere partition of $X$ into: geodesics $X_{\alpha}$ indexed over a set $Q$, and a set $\mathcal{Z}$ on which $f=0$ $\mathfrak{m}$-a.e. 

The partition of $X$ into \emph{transport rays} $\sqcup X_\alpha=:\mathcal{T}$ and $\mathcal{Z}$ induced by the function $f$ determines the following disintegration formula: 
 \begin{equation}\label{eq:disintegration2}
\mathfrak{m}(B)=\int_{Q} h_{\alpha} \mathcal{H}^{1} \llcorner X_{\alpha}(B)\,  {\mathfrak{q}}(\mathrm{d} \alpha),\, \quad \forall B\subset\mathcal{T}.
 \end{equation}
 The probability measure $\mathfrak{q}$ in \eqref{eq:disintegration2}, defined on the set of indices $Q$, is obtained in a natural way from the essential partition $(X_{\alpha})_{\alpha \in Q}$ of $X$, roughly by projecting $\mathfrak{m}$  on the set $Q$ of equivalence classes (we refer to \cite{CM17, CM20} for the details). Moreover, we also have that for $\mathfrak{q}$-a.e. $\alpha\in Q$, $\int_{X_\alpha} f\,h_\alpha\mathcal{H}^1\llcorner X_\alpha = 0$

The key property is that, if $(X,\mathrm{d},\mathfrak{m})$ is an essentially non-branching $\mathrm{CD}(K,N)$ metric measure space, then each $h_{\alpha}$ is a $\mathrm{CD}(K,N)$ density over the ray $X_{\alpha}$ (see \cite[Theorem 3.6]{CM20}), i.e., if $N>1$
\begin{equation}\label{eq:localprelA}
(\log h_{\alpha})''\leq - K - \frac{1}{N-1} \big( (\log h_{\alpha})' \big)^{2},
\end{equation}
in the sense of distributions and point-wise except countably many points, compare with \cite[Lemma A.3, Lemma A.5, Proposition A.10]{CavallettiMilmanCD} (when $N=1$, $h_\alpha$ is constant). 
Equivalently, if $N>1$
\begin{equation}\label{eq:localprel}
\left(h_{\alpha}^{\frac{1}{N-1}}\right)''+\frac{K}{N-1}h_{\alpha}^{\frac{1}{N-1}}\leq 0\, ,
\end{equation}
in the sense of distributions. This amounts to say that the Curvature-Dimension condition of the ambient space $(X,\mathrm{d},\mathfrak{m})$ is inherited by the needles of the partition induced by $f$.
\medskip

Now, the Lévy--Gromov isoperimetric inequality \cref{LGBG} is obtained as in the following sketch, see \cite[Theorem 6.4]{CM17} for the more general statement in the category of $\mathrm{CD}$ spaces. Assume $\mathrm{Ric}\geq (n-1)g$ on $M$, and fix $t\in (0,1)$. For every Borel set $E$ with $|E|=t|M|$, consider $f:=\chi_E - t|M|$, and notice that $\int_M f\mathrm{d}\mathrm{vol}=0$. Using the disintegration of the measure as above (with guiding function $f$), and recalling that $\mathcal{M}_{(X,\mathrm{d},\mathfrak{m})}$ is the Minkowski content of the metric measure space $(X,\mathrm{d},\mathfrak{m})$, and $A_\varepsilon:=\{x:\mathrm{d}(x,A)\leq \varepsilon\}$, we get
\begin{equation}\label{eqn:BoundMinLG}
\mathcal{M}_{(M^n,\mathrm{d},|\cdot|)}(E)=\liminf_{\varepsilon\to 0}\frac{|E_\varepsilon|-|E|}{\varepsilon} \geq \int_Q \mathcal{M}_{(X_\alpha,\mathrm{d},h_\alpha\mathcal{H}^1)}(A_\alpha)\mathfrak{q}(\mathrm{d}\alpha),
\end{equation}
where $A_\alpha:=X_\alpha\cap A$, and $(h_\alpha\mathcal{H}^1)(A_\alpha)=t|M|$. Finally, since $(X_\alpha,\mathrm{d},h_\alpha\mathcal{H}^1)$ is a one-dimensional manifold where $h_\alpha$ satisfies \eqref{eq:localprel} with $K=n-1$, and $N=n$ (which means that it satisfies the weak curvature dimension condition $\mathrm{CD}(n-1,n)$), we can use Milman's result in \cite[Theorem 1.2]{MilmanJEMS} to bound $\mathcal{M}_{(X_\alpha,\mathrm{d},h_\alpha\mathcal{H}^1)}(A_\alpha)$ from below with the value in the model space and eventually conclude the proof.  

Finally, to pass from the Minkowski content to the perimeter in \eqref{eqn:BoundMinLG}, we need to use the fact that the lower semicontinuous relaxation of $\mathcal{M}(\cdot)$ with respect to the convergence in measure (on finite measure sets) is the perimeter, see again \cite{ADMG17}. Thus from \eqref{eqn:BoundMinLG} we get 
\[
I(t|E|)\geq P(E)\geq \frac{|M|}{|\mathbb S^n|}I_{\mathbb S^n}(t|\mathbb S^n|),
\]
as desired.
}

\subsection{Connections to the stable Bernstein problem}

Using an unequally weighted isoperimetric profile, and carefully reworking the proof of \cref{BG} in \cref{LGBG}, one can prove the following stronger version of \cref{BG} in \cref{LGBG}. Denote $\mathrm{Ric}(x)$ the smallest eigenvalue of the Ricci tensor at $x$.
\begin{theorem}[A--X \cite{AX24}]\label{spectralBG}
    Let $(M^n,g)$ be a compact Riemannian manifold with $n\geq 3$. Assume that 
    \begin{equation}\label{eqn:SpectralBG}
    \int_M \left(\frac{n-1}{n-2}|\nabla\varphi|^2 + \varphi^2\mathrm{Ric}\right)\mathrm{d}\mathcal{H}^n \geq (n-1)\int_M \varphi^2\mathrm{d}\mathcal{H}^n, \qquad \forall \varphi\in C^\infty(M).
    \end{equation}
    Then 
    \[
    |M|\leq |\mathbb S^n|,
    \]
    and if equality holds, then $M$ is isometric to $\mathbb S^n$.
\end{theorem}
\begin{proof}
    Let us sketch the proof. From the assumption\footnote{Notice: the assumption can be rewritten as: $\lambda_1(-\frac{n-1}{n-2}\Delta+\mathrm{Ric})\geq n-1$, where $\lambda_1(\cdot)$ denotes the least eigenvalue of the operator $\cdot$.}, regularizing $\mathrm{Ric}$, and using classical PDE analysis, one gets that there is a function $u>0$\footnote{Actually $u\in C^{2,\alpha}$ for every $0<\alpha<1$.} with $\min u=1$ such that 
    \begin{equation}\label{eqn:ReducesTo}
    -\frac{n-1}{n-2}\Delta u + \mathrm{Ric}\cdot u \geq (n-1)u.
    \end{equation}
    Then, we define $I_M:(0,\int_M u^{\frac{2}{n-2}})\to \mathbb R$ to be
    \[
    I_M(v):=\inf\left\{\int_{\p^* E}u^{\frac{n-1}{n-2}}: \text{$E$ has finite perimeter, and} \int_E u^{\frac{2}{n-2}}=v\right\}.
    \]
    It can be proved, by carefully adjusting \textbf{Case 1}, and \textbf{Case 2} in the proof of \cref{thm:ConcavitySharp}, that 
    \begin{equation}\label{TheEquation}
    -I_M''I_M \geq (n-1)+\frac{(I_M')^2}{n-1},
    \end{equation}
    holds in the viscosity sense on the domain of $I_M$.
    Thus, by ODE comparison with the isoperimetric profile of the round sphere (which satisfies equality in \eqref{TheEquation}, see \cref{rem:Equality}) as in the proof of \cref{BG} in \cref{LGBG}, and by using $u\geq 1$, we get
    \[
    |M|\leq \int_M u ^{\frac{2}{n-2}} \leq |\mathbb S^n|.
    \]
    If equality holds then $u\equiv 1$, and thus \eqref{eqn:ReducesTo} reduces to $\mathrm{Ric}\geq n-1$. Hence, the rigidity is concluded by the rigidity in the classical Bishop--Gromov inequality.
\end{proof}
\begin{remark}[Sharpness of \cref{spectralBG}]
    For every $n\geq 3$, every $\gamma>\frac{n-1}{n-2}$, and every $V>0$, there exists a metric on $\mathbb S^1\times\mathbb S^{n-1}$ such that \eqref{eqn:SpectralBG} holds with $\gamma$ instead of $\frac{n-1}{n-2}$, and $|M|>V$ (A--X \cite{AX24}). Thus the constant $\frac{n-1}{n-2}$ is sharp in \cref{spectralBG}.
\end{remark}

The result in \cref{spectralBG} was first proved when $n=3$ in \cite{CLMS}, and has played a key role in the solution of the stable Bernstein problem in $\mathbb R^5$ and $\mathbb R^6$ (Chodosh--Li--Minter--Stryker \cite{CLMS}, and Mazet \cite{Mazet}).

\begin{theorem}[Fischer-Colbrie--Schoen \cite{FischerColbrieSchoenCPAM}, Do Carmo--Peng \cite{DoCarmoPeng79}, Pogorelov \cite{Pogorelov81}; Chodosh--Li \cite{ChodoshLi4, ChodoshLiAlternative4}, Catino--Mastrolia--Roncoroni \cite{CatinoMastroliaRoncoroni}; Chodosh--Li--Stryker--Minter \cite{CLMS}; Mazet \cite{Mazet}]\label{StableBernstein}
    Let $M^n\hookrightarrow \mathbb R^{n+1}$ be an oriented, complete, connected, two-sided stable minimal hypersurface. If $n\leq 5$, then $M$ is a Euclidean hyperplane.
\end{theorem}

Let us roughly describe (one of the approaches to) the recent proofs of \cref{StableBernstein} when $n\in \{3,4,5\}$. For a direct approach, which does not involve directly biRicci curvature, the reader might also read the computations \cite{AntonelliXuNote} on my (and Kai Xu's) webpage. First, given $\alpha>0$, and $u,v\in T_pM$, let us define 
\[
\mathrm{biRic}_p^{\alpha}(u,v):= \mathrm{Ric}_p(u,u)+\alpha\mathrm{Ric}_p(v,v)-\alpha\mathrm{Sect}_p(u\wedge v),
\]
and 
\[
\mathrm{biRic}^\alpha(p) := \min_{\text{$u,v\in T_pM$ orthonormal}}\mathrm{biRic}_p^{\alpha}(u,v).
\]
The proof (roughly) goes as follows. There is some topology to be taken care of, but we will be sketchy for the sake of readability.
\begin{enumerate}
    \item Without loss of generality, let $\mathbb R^n \ni 0 \in M$. Let $r$ denote the Euclidean distance from $0$. The Gulliver--Lawson conformal change $\tilde g:=r^{-2}g$ satisfies the spectral condition:
    \begin{equation}\label{eqn:SpectralAlphaBiric}
    -\gamma\Delta_{(M,\tilde g)} + \mathrm{biRic}^{\alpha}_{(M,\tilde g)} \geq \delta>0,
    \end{equation}
    for some carefully chosen $\gamma,\alpha,\delta \in (0,+\infty)$.
    \item The control \eqref{eqn:SpectralAlphaBiric} has the following consequence, in the metric $\tilde g$. There exist $\gamma',\varepsilon,C\in (0,\infty)$ for which we have the following. 
    
    For every $\rho>0$ there exists a compact hypersurface $\Sigma^{n-1}\subset\subset M\setminus B^{g}_\rho(0)$ such that $d_{\tilde g}(\partial B^g_\rho(0),\Sigma)\leq C$ and
    \begin{equation}\label{SpectralRicci}
    -\gamma'\Delta_{(\Sigma,\tilde g)}+\mathrm{Ric}_{(\mathrm{\Sigma},\tilde g)}\geq \varepsilon>0.
    \end{equation}
    This $\Sigma$ is produced by taking a carefully chosen $\mu$-bubble, i.e., a minimizer of the energy, computed with respect to $\tilde g$ (in a properly chosen annulus outside $B_\rho^{g}(0)$)
    \[
    E(\Omega):=\int_{\partial\Omega} u^{\gamma}-\int hu^{\gamma},
    \]
    for a properly chosen function $h$, and where $u>0$ is a smooth function satisfying $-\gamma\Delta_{(M,\tilde g)}u+\mathrm{biRic}^{\alpha}_{(M,\tilde g)}u-\delta u \geq 0$ (which exists by \eqref{eqn:SpectralAlphaBiric}). 

    The game is to get the smallest possible $\gamma'$ for which there are positive $\varepsilon,C>0$ such that one can ensure \eqref{SpectralRicci}. This reduces at finding the best spectral $\alpha$-biRicci condition on the Gulliver--Lawson conformal change \eqref{eqn:SpectralAlphaBiric}.
    \item If we succeeded in obtaining $\gamma'\leq \frac{n-2}{n-3}$, we can apply \cref{spectralBG} and infer that $|\tilde\Sigma^{n-1}|_{\tilde g}\leq C(\gamma',\varepsilon,n)$ for every connected component $\tilde\Sigma$ of $\Sigma$. 
    \item Going back to the original metric, and up to possibly passing to the universal cover, one has the following. For every $r>0$ there exists $\Omega^n\supset B_r^g(0)$ such that $\partial\Omega^n$ is connected, $d_g(\partial\Omega^n,\partial B_r^g(0))\leq Cr$ and
    \[
    |\partial\Omega^n|_g\leq C r^{n-1},
    \]
    for a universal $C$. Then Michael--Simon's isoperimetric inequality on $M$ (\cite{MichaelSimon}, \cite{BrendleIsop}) gives that $|B_r^g(0)|_g\leq Cr^n$ for every $r>0$. Hence $(M,g)$ has intrinsic Euclidean volume growth and the work of Schoen--Simon--Yau \cite{SSY} implies that $M$ is flat.
\end{enumerate}
\begin{remark}[Limitation of the strategy to solve the stable Bernstein problem]
    Unfortunately, it can be proved that, when $n=6$, the smallest $\gamma'$ one can get with the previous approach is $>4/3$. See, e.g., \cite{AntonelliXuNote}, and \cite{TamNote}. Hence, the strategy has a limitation in proving that an oriented, connected, complete, two-sided, stable minimal hypersurface $M^6\hookrightarrow \mathbb R^7$ has Euclidean volume growth. Another complication is that Schoen--Simon--Yau's result is not available for stable minimal hypersurfaces in $\mathbb R^7$. Anyway, recently Bellettini \cite{Bellettini} was able to carry out Schoen--Simon--Yau's result for (\textit{properly immersed}) stable minimal hypersurfaces in $\mathbb R^7$ having \textit{extrinsic} Euclidean volume growth.

    $n=6$ is the only remaining case where the stable Bernstein problem is not settled. Indeed, as a consequence of Bombieri--De Giorgi--Giusti's work \cite{BDG69}, we know that there exists a connected, complete, two-sided, stable minimal hypersurface $M^7\hookrightarrow \mathbb R^8$ that is not a Euclidean hyperplane.
\end{remark}

\begin{question}
    Let $M^6\hookrightarrow \mathbb R^7$ be a connected, complete, two-sided, stable minimal hypersurface. Understand whether $M$ must be a Euclidean hyperplane.
\end{question}

\newpage

\section{Third Lecture}

\subsection{Gromov--Hausdorff convergence} Gromov in the '80s, while proving that a finitely generated group has polynomial growth if and only if it is virtually nilpotent \cite{GromovPolynomial}, introduced a flexible notion of convergence of metric spaces. There are several ways one can introduce pGH (pointed Gromov--Hausdorff) convergence, or pmGH (pointed measured Gromov--Hausdorff) convergence (a variant in which also the measure converges). We refer to the book Burago--Burago--Ivanov \cite{BuragoBuragoIvanovBook} or to the work by Gigli--Mondino--Savaré \cite{GigliMondinoSavare15}. We adopt an extrinsic point of view. All metric spaces are complete and separable.

\begin{definition}[pGH and pmGH convergence]\label{def:GHconvergence}
A sequence $\{ (X_i, \mathrm{d}_i, x_i) \}_{i\in \mathbb N}$ \emph{pGH-converges} to a pointed metric space $ (Y, \mathrm{d}_Y, y)$ if there exist a complete separable metric space $(Z, \mathrm{d}_Z)$ and isometric embeddings
\[
\begin{split}
&\Psi_i:(X_i, \mathrm{d}_i) \to (Z,\mathrm{d}_Z), \\
&\Psi:(Y, \mathrm{d}_Y) \to (Z,\mathrm{d}_Z),
\end{split}
\]
such that for any $\varepsilon,R>0$ there is $i_0(\varepsilon,R)\in\mathbb N$ such that
\[
\Psi_i(B_R^{X_i}(x_i)) \subset \left[ \Psi(B_R^Y(y))\right]_\varepsilon,
\qquad
\Psi(B_R^{Y}(y)) \subset \left[ \Psi_i(B_R^{X_i}(x_i))\right]_\varepsilon,
\]
for any $i\geq i_0$, where $[A]_\varepsilon:= \{ z\in Z : \mathrm{d}_Z(z,A)\leq \varepsilon\}$ for any $A \subset Z$.

Let $\mathfrak{m}_i$ and $\mu$ be Radon measures on $X_i,Y$. If in addition we also have $(\Psi_i)_\sharp\mathfrak{m}_i \rightharpoonup \Psi_\sharp \mu$ with respect to duality with continuous bounded functions with bounded support on $Z$, then the convergence is said to hold in the \emph{$\mathrm{pmGH}$ sense}.
\end{definition}

A metric space $(X,\mathrm{d})$ is said to be \emph{locally doubling} if, for every $R>0$, there is $C(R)$ such that 
\begin{equation}\label{eqn:LocallyDoubling}
\text{For all $x\in X$, and $r<R$, $B_{2r}(x)$ can be covered by $\leq C(R)$ balls of radius $r$.}
\end{equation}

\begin{theorem}[Gromov]
    Let $\mathcal{F}$ be a class of metric spaces that is uniformly locally doubling. Namely, for every $R>0$, there exists $C(R)$ such that \eqref{eqn:LocallyDoubling} holds for every $(X,\mathrm{d})\in\mathcal{F}$. Then, any sequence $(X_i,\mathrm{d_i},x_i)$, where $(X_i,\mathrm{d}_i)\in\mathcal{F}$, and $x_i\in X_i$, is pre-compact in the pGH topology.

    In particular\footnote{As a consequence of \cref{RicLocDoub}.}, given $N\in\mathbb N$, and $k\in\mathbb R$, the class $\{(M^n,g,p):n\leq N, \mathrm{Ric}\geq k\cdot g, p\in M\}$ is precompact in the pGH-topology. Any limit point is called \emph{Ricci-limit space}.
\end{theorem}

Ricci-limit spaces have been thoroughly investigated starting from the 90s (after the seminal works of Cheeger--Colding \cite{ChCo0, ChCo1}), and inspired developments in non-smooth geometry, in particular CD and RCD spaces (Sturm \cite{Sturm1, Sturm2}, Lott--Villani \cite{LottVillani}, Ambrosio--Gigli--Savaré \cite{AGS1Inve, AGSDuke}, Gigli \cite{Gigli12}, ...). The theory is incredibly vast and has found a lot of applications recently, see \cref{BMandRCD} for a quick introduction and the surveys \cite{AmbrosioSurvey, DeGiorgiGromov}. Let us recall the following.

\begin{theorem}[Volume convergence, Colding and Cheeger--Colding \cite{ChCo1, Colding97}]\label{colding}
    Let $n\geq 2$, and $k\in\mathbb R$. Let $(M_i^n,g_i,p_i)$ be a sequence of pointed smooth complete Riemannian manifolds with $\mathrm{Ric}_{M_i}\geq k\cdot g_i$, and $\inf |B_1(p_i)| >0$. Then, up to subsequences 
    \[
    (M^n_i,\mathrm{d}_i,p_i,\mathcal{H}^n)\longrightarrow (X,\mathrm{d},p,\mathcal{H}^n), \quad \text{in the pmGH sense}.
    \]
\end{theorem}

\subsubsection{Isoperimetric problem and concentration--compactness}
We can use the theory of Ricci-limit spaces to detect what happens when we miss to have existence of isoperimetric sets in a non-compact manifold with Ricci curvature bounded from below. We state the following concentration--compactness theorem only for Riemannian manifolds with $\mathrm{Ric}\geq 0$. Analogous versions hold when $\mathrm{Ric}\geq k\cdot g$, or even in the non-smooth setting (finite dimensional Alexandrov or RCD spaces). We omit the proof, which can be found in A--F--Po \cite{AFP21} or A--Na--Po \cite{AntonelliNardulliPozzetta}. It is based on the broad principle of concentration-compactness, first put forward in Ritoré--Rosales \cite{RitRosales04}, Nardulli \cite{Nar14}, and then exploited in Mondino--Nardulli \cite{MondinoNardulli16}. The following theorem holds more in general for finite dimensional $\mathrm{RCD}$ spaces, see \cite{AntonelliNardulliPozzetta}.

\begin{theorem}[(A particular case of) A--Na--Po \cite{AntonelliNardulliPozzetta}]\label{concentrationcompactness}
    Let $(M^n,g)$ be a smooth complete non-compact Riemannian manifold with $\mathrm{Ric}\geq 0$, and $\inf_{x\in M}|B_1(x)|>0$. Let $v>0$. Then at least one of the following two hold.
    \begin{enumerate}
        \item There is an isoperimetric set with volume $v$ in $M$.
        \item There is a sequence of diverging points $x_i\in M$ and a pointed Ricci-limit space $(X,\mathrm{d},x)$ such that 
        \[
        (M,\mathrm{d},x_i,\mathcal{H}^n)\longrightarrow (X,\mathrm{d},x,\mathcal{H}^n), \qquad \text{in the pmGH sense},
        \]
        and there is an isoperimetric set $E\subset X$ with $|E|_X=v$ such that 
        \[
        P_X(E)=I_X(v)=I_M(v).
        \]
        
    \end{enumerate}
\end{theorem}

Notice that if we remove the hypothesis $\inf_{x\in M}|B_1(x)|>0$ in \cref{concentrationcompactness}, the isoperimetric problem on $(M^n,g)$ trivializes, see \cref{exTrivialI}. Compare this with \cref{Questionk-1}.

\subsection{Weak mean curvature and sharp concavity in the non-compact case}

At this point we are ready to complete the proof of \textbf{Case 3} in \cref{thm:ConcavitySharp}. We say that $(X,\mathrm{d})$ is an \emph{$n$-dimensional non-collapsed Ricci-limit space with $\mathrm{Ric}\geq 0$} if there are $(M^n_i,g_i,p_i)$, with $\mathrm{Ric}\geq -\varepsilon_i\cdot g_i$, $\varepsilon_i\to 0$, and $\inf|B_1(p_i)|>0$ such that 
\[
(M^n_i,\mathrm{d}_i,p_i)\to (X,\mathrm{d},x\in X), \qquad\text{in the pGH sense}.
\]
Hence, by \cref{colding}, also the volume measures converge. The following result is taken from A--G \cite{AntonelliGlaudo}, and it is a re-elaboration of the results of A--Pa--Po--S \cite{ConcavitySharpAPPS}. Again, we focus on Ricci-limit spaces here but the result can be proved for $n$-dimensional $\mathrm{RCD}(0,n)$ spaces $(X,\mathrm{d},\mathcal{H}^n)$ as well.

\begin{theorem}[A--G \cite{AntonelliGlaudo} after A--Pa--Po-S \cite{ConcavitySharpAPPS, ConcavitySharpAPPS2}]\label{MeanCurvatureMonotonicity}
    Let $(X,\mathrm{d})$ be a non-compact $n$-dimensional non-collapsed Ricci-limit space with $\mathrm{Ric}\geq 0$, and let $E\subset X$ be an isoperimetric set\footnote{By A--Pa--Po \cite{AntonelliPasqualettoPozzetta21} there is an open set $\tilde E$ such that $|\tilde E\Delta E|=0$, and $\partial^*\tilde E=\partial \tilde E$. Without loss of generality we always consider this representative.}. Let
    \begin{equation*}
        E_r :=
        \begin{cases}
            \{x\in X: \mathrm{d}_{\overline E}(x)\leq r\}
            &\quad\text{if $r \geq 0$,} \\
            \{x\in X: \mathrm{d}_{X\setminus E}(x)\geq -r\} 
            &\quad\text{if $r < 0$},
        \end{cases}
    \end{equation*}
    where $\mathrm{d}_F:X\to[0,\infty)$ denotes the distance from the set $F$.
    
    There exists $H\geq 0$\footnote{$H$ agrees with the (constant) mean curvature of the regular part of $\partial E$, whenever $X$ is a Riemannian manifold. Moreover, if $I_X$ is differentiable at $|E|$, we have $H=I'_X(|E|)$.},
    such that the function
    \begin{equation}\label{Bellaaa}
        \mathbb R\ni r\mapsto P_X(E_r)^{\frac{n}{n-1}}-\frac{n}{n-1}HP_X(E)^{\frac{1}{n-1}}|E_r|,
    \end{equation}
    achieves its maximum at $r=0$ on $(-\infty,\infty)$\footnote{One can show that $E_r=\emptyset$ for every $r\leq -\frac{n-1}{H}$. Hence, for every $r\leq -\frac{n-1}{H}$ the previous function is constantly equal to zero.}. Moreover, for almost every $r\geq s\geq  0$, we have
    \begin{equation}\label{eq:enlargement-ratio}
        \frac{P_X(E_r)^{\frac{n}{n-1}}}{|E_r|} \leq 
        \frac{P_X(E_s)^{\frac{n}{n-1}}}{|E_s|},
    \end{equation}
    and the function 
    \[
    [0,\infty)\ni r\mapsto |E_r|^{1/n}, \qquad \text{is concave}.
    \]
\end{theorem}
\begin{proof}
    Let us present a sketch. The main technically demanding point is the following: By A--Pa--Po--S \cite[Theorem 3.3]{ConcavitySharpAPPS} we get that there exists $H\geq 0$ such that
    \begin{equation}\label{eqn:Distributional}
    \Delta \mathrm{d}_{\overline{E}} \leq \frac{H}{1+\frac{H}{n-1}\mathrm{d}_{\overline{E}}}
    \end{equation}
    holds in the (appropriate) distributional (or, equivalently, viscosity) sense on $X\setminus\overline{E}$\footnote{The proof of the latter property is inspired by ideas of Caffarelli--Cordoba \cite{CaffarelliCordoba93}, Petrunin \cite{Petruninharmonic}, and Mondino--Semola \cite{MoS21}. Similar ideas appeared in \cite{WhiteBup}. See \cref{MCB} for a detailed road-map of the proof.}. Observe: \eqref{eqn:Distributional} is a weak way to encode the evolution equation, which holds when $X,\partial E$ are smooth,
    \begin{equation}\label{eqn:ControlHt}
    \frac{\mathrm{d}}{\mathrm{d}t} H_t = -|\mathrm{II}_t|^2 - \mathrm{Ric}(\nu_t,\nu_t) \leq -\frac{H_t^2}{n-1},
    \end{equation}
    for $t\in [0,\varepsilon)$, where $H_t,\mathrm{II}_t,\nu_t$ denote the mean curvature, second fundamental form, and unit outward-pointing normal of $\partial E_t$. 
    
    Integrating \eqref{eqn:Distributional} on a slab $E_R\setminus E_r$, for $0<r<R$, using Gauss--Green, and sending $R\to r$ we get 
    \begin{equation}\label{eqn:Distributional2}
    \frac{\mathrm{d}}{\mathrm{d}t}P(E_t)\leq P(E_t)\frac{H}{1+\frac{H}{n-1}t}, \qquad \text{distributionally on $(0,\infty)$}.
    \end{equation}
    From \eqref{eqn:Distributional2} we get that the function in \eqref{Bellaaa}, restricted to $[0,+\infty)$, achieves its maximum at $r=0$ (actually, it's non-increasing), and \eqref{eq:enlargement-ratio} holds, see \cref{ExComputationsAG}. Finally, by exploiting the analogue of \eqref{eqn:Distributional} inside of $E$, we also get that the function in \eqref{Bellaaa}, restricted to $(-\infty,0]$, achieves its maximum at $r=0$, as desired.

    To infer the last concavity, just notice that, for almost every $r\in [0,\infty)$,
    \[
    \frac{\mathrm{d}}{\mathrm{d}r} |E_r|^{1/n} = \frac{1}{n}\frac{P(E_r)}{|E_r|^{\frac{n-1}{n}}} = \frac{1}{n} \left(\frac{P(E_r)^{\frac{n}{n-1}}}{|E_r|}\right)^{\frac{n-1}{n}},
    \]
    which is non-increasing due to \eqref{eq:enlargement-ratio}.
\end{proof}
\begin{proof}[Proof of Case 3 of \cref{thm:ConcavitySharp}]
   We give the proof in the case $k=0$ for simplicity. Observe that the assertion we want to show is equivalent to showing that $(I^{\frac{n}{n-1}})''\leq 0$ holds in the viscosity sense on $(0,|M|)$, namely that $I^{\frac{n}{n-1}}$ is concave.
   
   Fix $v>0$. If there is an isoperimetric set with volume $v$ we can argue as in \textbf{Case 1} (or \textbf{Case 2}) of the proof of \cref{thm:ConcavitySharp}. If not, by \cref{concentrationcompactness}, there is a non-collapsed Ricci limit space at infinity $(X,\mathrm{d},x)$ and an isoperimetric set $E\subset X$, with $|E|_X=v$, such that
   \[
   P_X(E)=I_X(|E|_X)=I_M(|E|).
   \]
   Denote
   \[
   f(x):=P_X(E)^{\frac{n}{n-1}}-\frac{n}{n-1}HP_X(E)^{\frac{1}{n-1}}|E|+\frac{n}{n-1}HP_X(E)^{\frac{1}{n-1}}x,
   \]
   and observe that $I_X\geq I_M$, see \cref{IXIM}. Hence, by \cref{MeanCurvatureMonotonicity} we have, for every $r\in (-\infty,\infty)$, 
   \[
   I_M(|E_r|)^{\frac{n}{n-1}} \leq I_X(|E_r|)^{\frac{n}{n-1}} \leq P_X(E_r)^{\frac{n}{n-1}} \leq f(|E_r|),
   \]
   and $I_M(|E|)^{\frac{n}{n-1}}=I_X(|E|)^\frac{n}{n-1}=f(|E|)$. Thus $f$ is an affine function touching $I_M^{\frac{n}{n-1}}$ from above at $|E|=v$, and then $(I^{\frac{n}{n-1}})''\leq 0$ in the viscosity sense at $v$, as desired.
   \end{proof}

{
\subsubsection{Mean curvature barriers}\label{MCB}

We give a detailed road-map of the proof of \eqref{eqn:Distributional}. Let us recall some notions that are useful. The presentation here draws heavily from A--Pa--Po--S \cite{ConcavitySharpAPPS}.
\begin{enumerate}
    \item The \emph{Cheeger energy} (\cite{AGS1Inve} after \cite{CheegerGAFA}) {on an $\mathrm{RCD}(K,N)$ metric measure space} \((X,\mathrm{d},\mathfrak{m})\) is defined as 
\[
\mathrm{Ch}(f):=\inf\bigg\{\liminf_{n\to\infty} \frac{1}{2}\int |\nabla f_n|^2\mathrm{d}\mathfrak{m}\;\bigg|\;(f_n)_{n\in\mathbb N}\subseteq{\rm LIP}_{\mathrm{bs}}(X),\,f_n\to f\text{ in }L^2(X)\bigg\}\, .
\]
Here $|\nabla f|(x):=\limsup_{y\to x}\frac{|f(y)-f(x)|}{\mathrm{d}(x,y)}$. The \emph{Sobolev space} \(H^{1,2}(X)\) is defined as the finiteness domain \(\{f\in L^2(X)\,:\,\mathrm{Ch}(f)<+\infty\}\) of the Cheeger energy.
With a slight abuse of notation, when $f\in H^{1,2}(X)$, \(\mathrm{Ch}(f)=\frac{1}{2}\int|\nabla f|^2\mathrm{d}\mathfrak{m}\),
for a uniquely determined function \(|\nabla f|\in L^2(X)\).

Moreover, given an open $\Omega\subset X$, we declare that a given function \(f\in L^2(\Omega)\)
belongs to the \emph{local Sobolev space} \(H^{1,2}(\Omega)\) provided \(\eta f\in H^{1,2}(X)\) for every \(\eta\in{\rm LIP}_{\mathrm{bs}}(\Omega)\) and {the function}
\[
|\nabla f|:={\rm ess\,sup}\big\{\chi_{\{\eta=1\}}|\nabla(\eta f)|\;\big|\;\eta\in{\rm LIP}_{\mathrm{bs}}(\Omega)\big\}
\]
is in $L^2(X)$.
\item For an open set $\Omega\subset X$ We define the bilinear mapping \(H^{1,2}(\Omega)\times H^{1,2}(\Omega)\ni(f,g)\mapsto\nabla f\cdot\nabla g\in L^1(\Omega)\) as
\[
\nabla f\cdot\nabla g:=\frac{|\nabla(f+g)|^2-|\nabla f|^2-|\nabla g|^2}{2},\quad\text{ for every }f,g\in H^{1,2}(\Omega)\, .
\]
\item We say that a function
\(f\in H^{1,2}(\Omega)\) has \emph{measure-valued Laplacian} in \(\Omega\), \(f\in D(\mathbf\Delta,\Omega)\) for short, provided there
exists a (uniquely determined) locally finite measure \(\mathbf\Delta f\) on \(\Omega\) such that
\[
\int_\Omega g\mathbf\Delta f:=\int_\Omega g\mathrm{d}\mathbf\Delta f=-\int_\Omega\nabla g\cdot\nabla f\mathrm{d}\mathfrak{m}\, ,\quad\text{ for every }g\in{\rm LIP}_{\mathrm{bs}}(\Omega)\, .
\]
Moreover, given functions \(f\in{\rm LIP}(\Omega)\cap H^{1,2}(\Omega)\)
and \(\eta\in C_b(\Omega)\), we say that \emph{\(\mathbf\Delta f\leq\eta\) in the distributional sense} if \(f\in D(\mathbf\Delta,\Omega)\) and
\(\mathbf\Delta f\leq\eta\mathfrak{m}\) as measures.
\item Let $\Omega\subset X$ be an open and bounded domain. Let $f:\Omega\to\mathbb R$ be locally Lipschitz and $\eta\in\mathrm{C}_b(\Omega)$. We say that $\Delta f\leq \eta$ in the \textit{viscosity sense} in $\Omega$ if the following holds. For any $\Omega'\Subset\Omega$ and for any test function $\phi:\Omega'\to\mathbb R$ such that 
\begin{itemize}
\item[(i)] $\phi\in D(\mathbf\Delta, \Omega')$ and $\Delta\phi$ is continuous on $\Omega'$;
\item[(ii)] for some $x\in \Omega'$ we have 
$\phi(x)=f(x)$ and $\phi(y)\leq f(y)$ for any $y\in\Omega'$, $y\neq x$;
\end{itemize}
we have
\begin{equation*}
\Delta \phi(x)\leq \eta(x)\, .
\end{equation*} 
{ A function $\phi$ as in items (i) and (ii) above will be called {\em lower supporting function of $f$}. When instead of $\leq$ we consider $\geq$ in the definition above, a function $\phi$ as in items (i) and (ii) will be called {\em upper supporting function of $f$}.}
\item Bounds for the Laplacian in the viscosity and distributional sense are equivalent in the smooth Riemannian setting. This is true also for $\mathrm{RCD}$ spaces, see \cite[Theorem 3.24]{MoS21} and \cite{GMSPams}. In particular the following holds.

Let $\Omega\subset X$ be an open and bounded domain, $f:\Omega\to\mathbb R$ be a Lipschitz function and $\eta:\Omega\to\mathbb R$ be continuous. Then $\boldsymbol{\Delta} f\leq \eta$ in the sense of distributions if and only if $\Delta f\leq \eta$ in the viscosity sense.
\end{enumerate}

We are now ready for the proof of the fact that it exists $H\geq 0$ such that \eqref{eqn:Distributional} holds in \cref{MeanCurvatureMonotonicity}.

\begin{proof}[Proof of \eqref{eqn:Distributional}]
The proof will be divided into two steps. Let $f:=\mathrm{d}_{\overline E}$. In the first one we are going to prove the weaker adimensional bounds
\begin{equation}\label{eq:adimensional laplacian comparison}
\boldsymbol{\Delta} f\geq H\,\quad\text{on $E$ and }\quad \boldsymbol{\Delta} f\leq H\, \quad\text{on $X\setminus \overline{E}$}\, ,
\end{equation}
corresponding to the limit as $n\to \infty$ of the bounds in \eqref{eqn:Distributional}. In the second part we will show how to obtain the sharp dimensional bounds with an application of the localization technique.

\medskip

\textbf{Step 1.} 
By Items (4) and (5) in the list at the beginning of \cref{MCB}, {\eqref{eq:adimensional laplacian comparison}} is equivalent to the following claim.

\textbf{Claim}: The supremum of the values of the Laplacians of lower supporting functions of $f$ 
at touching points on $X\setminus\overline{E}$ is lower than the infimum of the values of Laplacians of upper supporting functions of $f$ at touching points on $E$. Thus, letting $H$ be any value between the supremum and the infimum of the two sets, then $\Delta f\leq H$ holds on $X\setminus \overline{E}$ in the viscosity sense and $\Delta f\geq H$ holds on $E$, again in the viscosity sense. Thus, by item (5) in the list at the beginning of \cref{MCB}, \eqref{eq:adimensional laplacian comparison} also holds in the sense of distributions.
\medskip

Let us prove the claim. We argue by contradiction. If it is not true, by a truncation argument and by possibly adding/subtracting multiples of the squared distance function, we can find $x\in X\setminus \overline{E}$, $y\in E$, $\lambda>0$, $\delta>0$ (that we think to be very small, {in particular, $\lambda<f(x)$ and $\delta<\lambda$}) and supporting functions $\overline{\psi}:X\to\mathbb R$ and $\overline{\chi}:X\to\mathbb R$ with the following properties:
\begin{itemize}
\item[ia)] $\overline{\psi}:X\to\mathbb R$ is Lipschitz and 
{it belongs to the domain of the measure-valued}
Laplacian on $B_{\lambda}(x)$;
\item[iia)] $\overline{\psi}(x)=f(x)$;
\item[iiia)] $\overline{\psi}(z)< f(z)$ for any $z\neq x$ and $\overline{\psi}<f-\delta$ on $X\setminus B_{\lambda}(x)$;
\end{itemize}
and 
\begin{itemize}
\item[ib)] $\overline{\chi}:X\to\mathbb R$ is Lipschitz and 
{it belongs to the domain of the measure-valued}
Laplacian on $B_{\lambda}(y)$;
\item[iib)] $\overline{\chi}(y)=f(y)$;
\item[iiib)] $\overline{\chi}(z)>f(z)$ for any $z\neq y$ and $\overline{\chi}>f+\delta$ on $X\setminus B_{\lambda}(y)$.
\end{itemize}
Moreover, there exist $c\in\mathbb R$ and $\varepsilon>0$ such that
\begin{equation}\label{eq:visccon1}
{\boldsymbol{\Delta}}\overline{\psi}\geq c+\varepsilon\, \quad\text{on $B_{\lambda}(x)$}\, ,
\end{equation}
and 
\begin{equation}\label{eq:visccon2}
{\boldsymbol{\Delta}}\overline{\chi}\leq c-\varepsilon\, \quad\text{on $B_{\lambda}(y)$}\, .
\end{equation}
\medskip

We consider the transform of $\overline{\psi}$ through the Hopf--Lax duality and introduce $\phi:X\to\mathbb R$ by letting
\begin{equation}\label{eq:HLphi}
\phi(z):=\sup_{w\in X}\{\overline{\psi}(w)-\mathrm{d}(w,z)\}\, .
\end{equation}
Analogously, we let $\eta:X\to\mathbb R$ be defined by
\begin{equation}\label{eq:HLeta}
\eta(z):=\inf_{w\in X}\{\overline{\chi}(w)+\mathrm{d}(z,w)\}\, .
\end{equation}
\medskip

Let $X_{\Sigma}$ and $Y_{\Sigma}$ be the sets of touching points of minimizing geodesics from $x$ and $y$ respectively to $\Sigma:=\partial E$, i.e.
\begin{equation}
X_{\Sigma}:=\{w\in \partial E\, :\, f(x)-f(w)=\mathrm{d}(x,w)\},
\end{equation}
and
\begin{equation}
Y_{\Sigma}:=\{w\in \partial E\, :\, f(w)-f(y)=\mathrm{d}(y,w)\}\, .
\end{equation}
It is easy to verify that $X_{\Sigma}$ and $Y_{\Sigma}$ are compact subsets of $\partial E$. 

We get that $\phi\leq f$, $\eta\geq f$ {because $\overline{\psi}\leq f$ and $\overline{\chi}\geq f$ respectively}. Moreover, both $\phi$ and $\eta$ are $1$-Lipschitz functions, {because they are defined as suprema and infima of families of $1$-Lipschitz functions and they are finite at some point.}
Moreover, we claim that there exist neighborhoods $U_{\Sigma}\supset X_{\Sigma}$ and $V_{\Sigma}\supset Y_{\Sigma}$ such that: 
\begin{itemize}
\item[a)] $|\nabla\phi|=1$ holds $\mathcal{H}^n$-a.e. on $U_{\Sigma}$; 
\item[b)] $|\nabla\eta|=1$ holds $\mathcal{H}^n$-a.e. on $V_{\Sigma}$;
\item[c)] $\boldsymbol{\Delta}\phi \geq c+\varepsilon$ on $U_{\Sigma}$;
\item[d)] $\boldsymbol{\Delta}\eta\leq c-\varepsilon$ on $V_{\Sigma}$;
\item[e)] $\phi(z)=f(z)$ for any $z\in X$ such that $f(x)-f(z)=\mathrm{d}(x,z)$; 
\item[f)]$\eta(z)=f(z)$ for any $z\in X$ such that $f(z)-f(y)=\mathrm{d}(z,y)$.
\end{itemize} 
We leave the proof to the reader, a key tool being that the Laplacian bounds are preserved through the Hopf--Lax duality, see \cite[Theorem 4.9]{MoS21}. 
\smallskip

We can make the following reduction, whose details can be read in \cite{ConcavitySharpAPPS}. The original points $x$, $y$ in the contradiction argument can be taken as close as we wish to $\partial E$. Moreover they can be taken to have unique minimizing geodesics to $\partial E$, hence in particular unique footpoints on $\partial E$, that we shall denote by $x_E$ and $y_E$ respectively.
To finish the proof, we need to consider two cases. 

\textbf{Case 1: {$x_E=y_E$.}}

By construction, we have $\eta\geq f\geq\phi$. {Moreover, $\eta(x_E)=f(x_E)=\phi(x_E)$, by e), f).\\
Let us set $g:=\eta-\phi$. Observe that $g\geq 0$ and $g(x_E)=0$. Moreover, by c) and d), there exists a neighbourhood of $x_E$ where}
\begin{equation}
\boldsymbol{\Delta} g\leq (c-\varepsilon)+(-c-\varepsilon)\leq -2\varepsilon\, .
\end{equation} 
We get a contradiction, since $g$ would be a non-constant superharmonic function attaining its minimum at an interior point, see \cite[Theorem 2.8]{GigliRigoni}.

\medskip

\textbf{Case 2: {$x_E\neq y_E$}.}

{
Let us define
\begin{equation*}
E_{s,0}:=E\setminus \{\phi > s\}\, .
\end{equation*}
Observe that for $s=0$ we have $\{\phi>0\}\cap E=\emptyset$, since $\{\phi>0\}\subset \{f>0\}\subset X\setminus E$ by construction. When we decrease the value of $s$, the super-level set $ \{\phi >s\}$ starts cutting $E$.

Recall that $x_E\in \partial E$ is the footpoint of the minimizing geodesic from $x$ to $\overline{E}$. We claim that for any $s<0$ sufficiently close to $0$, $E_{s,0}$ is a perturbation of $E$ supported in a small ball $B_{r}(x_E)$, i.e. $\{\phi>s\}\cap E\subset B_r(x_E)$. We leave the verification of this claim to the reader.
\smallskip

A completely analogous verification shows that, for $0<t<\delta$, the set $E_{0,t}:=E\cup \{\eta\leq t\}$ is a perturbation of $E$ compactly supported in a small ball $B_r(y_E)$. When $r<\mathrm{d}(x_E,y_E)/2$, the interior and the exterior perturbations have disjoint supports. It is also elementary to check that the two perturbations are non-trivial.
\medskip
}

Let us set now, for any $-\delta\leq s\leq  0<t<\delta$,
\begin{equation}
E_{s,t}:=E\setminus \{\phi\geq s\}\cup\{\eta\leq t\}\, .
\end{equation}
We claim that there exist values $s,t$ in the range above such that 
\begin{equation}\label{eq:volpres}
\mathcal{H}^n(E_{s,t})=\mathcal{H}^n(E)
\end{equation}
and 
\begin{equation}\label{eq:perturb}
\mathcal{H}^n(E_{s,0})<\mathcal{H}^n(E)<\mathcal{H}^n(E_{0,t})\, .
\end{equation}
In order to establish the claim it is sufficient to prove that 
\begin{equation}
(s,t)\mapsto \mathcal{H}^n(E_{s,t}),
\end{equation}
is a continuous function. {Indeed, \eqref{eq:perturb} follows immediately from the non-triviality of the perturbations.} The sought continuity is a direct consequence of the $1$-Lipschitz regularity of $\phi$ and $\eta$, together with the compactness of the perturbations $E_{s,t}\Delta E$ and the properties $|\nabla\phi|=|\nabla\eta|=1$ $\mathcal{H}^n$-a.e., which guarantee that 
\begin{equation}
\mathcal{H}^n\left(\{\phi=s\}\cap U_{\Sigma}\right)=\mathcal{H}^n\left(\{\eta=t\}\cap V_{\Sigma}\right)=0\, ,
\end{equation}
for any $-\delta\leq s\leq 0<t<\delta$.

Given the claim, it is easy to find by a continuity argument $s$ and $t$ such that \eqref{eq:volpres} and \eqref{eq:perturb} hold. Moreover
letting $\Omega$ be the open neighbourhood of $\{x_E,y_E\}$ where the perturbation $E_{s,t}\Delta E$ is compactly contained, we have, by \cite[Proposition 6.1]{BPSGaussGreen} (see there for the precise explanation of the following intuitive notation),
\begin{equation}
\left(\nabla\phi\cdot\nu_{\{\phi<s\}}\right)_{\mathrm{int}}=\left(\nabla\phi\cdot\nu_{\{\phi<s\}}\right)_{\mathrm{ext}}=-1\, \quad \mathrm{Per}_{\{\phi<s\}}\text{-a.e. on $\Omega$},
\end{equation}
and 
\begin{equation}
\left(\nabla\eta\cdot\nu_{\{\eta<t\}}\right)_{\mathrm{int}}=\left(\nabla\eta\cdot\nu_{\{\eta<t\}}\right)_{\mathrm{ext}}=-1\, \quad \mathrm{Per}_{\{\eta<t\}}\text{-a.e. on $\Omega$}\, .
\end{equation}

We are going to reach a contradiction comparing the perimeter of $E_{s,t}$ with that of $E$. We estimate separately the differences in the perimeter coming from the perturbation induced by $\phi$ and $\eta$ by disjointedness:
\begin{equation}\label{eq:sepcontributions}
P(E_{s,t})-P(E)= \left(P(E_{s,0},\Omega)-P(E,\Omega)\right)+\left(P(E_{0,t},\Omega)-P(E,\Omega)\right)\, .
\end{equation}

The two contributions can be estimated as follows. Let us set $F:=E\cap \{\phi>s\}$ and $G:=\{\eta<t\}\setminus E$. Observe that by \eqref{eq:volpres} we have $\mathcal{H}^n(F)=\mathcal{H}^n(G)$.

On the one hand we can use the Gauss--Green formula and the results in \cite[Theorem 5.2, Proposition 5.4]{BPSGaussGreen}. Here $\mathcal{F}E$ denotes the reduced boundary (appropriately defined in the metric measure setting, see \cite{BPSGaussGreen}) of the set of finite perimeter $E$. We obtain
\begin{align}
\nonumber \int_{F^{(1)}}\boldsymbol{\Delta}\phi&=-\int_{\mathcal{F}F}\left(\nabla\phi\cdot\nu_F\right)_{\mathrm{int}}\mathrm{d} \mathrm{Per}\\
\nonumber &\leq-\int_{E^{(1)}\cap\mathcal{F}\{\phi>s\}}(\nabla \phi\cdot \nu_{\{\phi>s\}})_{\mathrm{int}}\mathrm{d} \mathrm{Per}-\int_{\mathcal{F}E\cap \{\phi>s\}^{(1)}}\left(\nabla\phi\cdot\nu_E\right)_{\mathrm{int}}\mathrm{d}\mathrm{Per}\\
\nonumber &=-\mathcal{H}^{n-1}\left(E^{(1)}\cap\mathcal{F}\{\phi>s\}\right)-\int_{\mathcal{F}E\cap \{\phi>s\}^{(1)}}\left(\nabla\phi\cdot\nu_E\right)_{\mathrm{int}}\mathrm{d}\mathrm{Per}\\
&\leq -\mathcal{H}^{n-1}\left(E^{(1)}\cap\mathcal{F}\{\phi>s\}\right)+\mathcal{H}^{n-1}\left(\mathcal{F}E\cap \{\phi>s\}^{(1)}\right)\, .\label{eq:boundgap1}
\end{align}

The analogous computation with $\nabla\eta$ in place of $\nabla \phi$ and $G$ in place of $F$ yields to
\begin{equation}\label{eq:boundgap2}
\int_{G^{(1)}}\boldsymbol{\Delta}\eta\geq \mathcal{H}^{n-1}\left(E^{(0)}\cap \mathcal{F}\{\eta<t\}\right)-\mathcal{H}^{n-1}\left(\mathcal{F}E\cap \{\eta<t\}^{(1)}\right)\, .
\end{equation}
\medskip

Now, the bounds on $\boldsymbol{\Delta}\phi$ and $\boldsymbol{\Delta}\eta$ imply 
\begin{equation}\label{eq:boundlap1}
\int_{F^{(1)}}\boldsymbol{\Delta}\phi\geq (c+\varepsilon)\mathcal{H}^n(F)\,
\end{equation}
and
\begin{equation}\label{eq:boundlap2}
\int_{G^{(1)}}\boldsymbol{\Delta}\eta\leq (c-\varepsilon)\mathcal{H}^n(G)\, .
\end{equation}
Hence, by \eqref{eq:sepcontributions}, \eqref{eq:boundgap1}, \eqref{eq:boundgap2}, \eqref{eq:boundlap1} and \eqref{eq:boundlap2}
\begin{align*}
P(E_{s,t})-P(E)=& \left(P(E_{s,0},\Omega)-P(E,\Omega)\right)+\left(P(E_{0,t},\Omega)-P(E,\Omega)\right)\\
= &\left(\mathcal{H}^{n-1}\left(E^{(1)}\cap\mathcal{F}\{\phi>s\}\right)- \mathcal{H}^{n-1}\left(\mathcal{F}E\cap \{\phi>s\}^{(1)}\right)\right)\\
&+  \left(\mathcal{H}^{n-1}\left(E^{(0)}\cap \mathcal{F}\{\eta<t\}\right)-\mathcal{H}^{n-1}\left(\mathcal{F}E\cap \{\eta<t\}^{(1)}\right)\right)\\
\leq &\int_{G^{(1)}}\boldsymbol{\Delta}\eta-\int_{F^{(1)}}\boldsymbol{\Delta}\phi\\
\leq& (c-\varepsilon)\mathcal{H}^n(G)-(c+\varepsilon)\mathcal{H}^n(F)=-2\varepsilon\mathcal{H}^n(F)<0\, ,
\end{align*}
yielding to the sought contradiction.
\medskip

\textbf{Step 2.}
Let us see how to pass from the adimensional estimates in \eqref{eq:adimensional laplacian comparison} to the sharp Laplacian comparison in \eqref{eqn:Distributional}.

We will rely on the localization technique from \cite{CM17,CM20}. See also \cref{Localization} for an explanation of the localization technique in a slightly different scenario.
From  \cite[Corollary 4.16]{CM20}, we know that 
\begin{equation}\label{eqn:LaplacianF}
\boldsymbol{\Delta} f \llcorner  X\setminus \overline{E}= (\boldsymbol{\Delta} f)^{\text{reg}}\llcorner  X\setminus \overline{E}+ (\boldsymbol{\Delta} f)^{\text{sing}}\llcorner  X\setminus \overline{E}\, ,
\end{equation}
 where the singular part $(\boldsymbol{\Delta} f)^{\text{sing}}\perp \mathcal{H}^{n}$ satisfies $(\boldsymbol{\Delta} f)^{\text{sing}} \llcorner  X\setminus \overline{E}\leq 0$ and the regular part $(\boldsymbol{\Delta} f)^{\text{reg}} \ll \mathcal{H}^{n}$ admits the representation formula
\begin{equation}\label{eq:repDeltafreg}
(\boldsymbol{\Delta} f)^{\text{reg}} \llcorner  X\setminus \overline{E} = (\log h_{\alpha})' \mathcal{H}^{n} \llcorner  X\setminus \overline{E}\,.
\end{equation}
In \eqref{eq:repDeltafreg}, $Q$ is a suitable set of indices,  $(h_{\alpha})_{\alpha\in Q}$  are suitable densities defined on geodesics $(X_{\alpha})_{\alpha \in Q}$, which are essentially partitioning $X\setminus \overline{E}$ (in the smooth setting, $(X_{\alpha})_{\alpha\in Q}$ correspond to the integral curves of $\nabla \mathrm{d}_{E}$), such that the following disintegration formula holds: 
 \begin{equation}\label{eq:disintegration}
 \mathcal{H}^{n}\llcorner  X\setminus \overline{E}=\int_{Q} h_{\alpha} \mathcal{H}^{1} \llcorner  X_{\alpha}\,  {\mathfrak{q}}(\mathrm{d}\alpha)\,.
 \end{equation}
 
 The key point for the proof of Step 2 is that 
\begin{equation}\label{eq:DiffEqCDKN}
(\log h_{\alpha})''\leq - \frac{1}{n-1} \big( (\log h_{\alpha})' \big)^{2},
\end{equation}
in the sense of distributions and point-wise except countably many points.
Equivalently
\begin{equation}\label{eq:halphaCD}
\left(h_{\alpha}^{\frac{1}{n-1}}\right)''\leq 0\, ,
\end{equation}
in the sense of distributions. This is true because the space $X$ is a Ricci-limit space, and thus an essentially non-branching CD space.
Moreover, from \eqref{eq:repDeltafreg} and \eqref{eq:adimensional laplacian comparison} we know that 
\begin{equation}\label{eq:halpha'Kdist}
(\log h_{\alpha})'(\mathrm{d}_{\overline{E}}) \leq H \, \text{ a.e. on $X_{\alpha}$, for $\mathfrak{q}$-a.e. $\alpha\in Q$}\, .
\end{equation}

The sharp estimates in \eqref{eqn:Distributional} then follow from the standard Riccati comparison applied to the functions $v(r):=\left(h_{\alpha}(r)/h_{\alpha}(0)\right)^{\frac{1}{N-1}}$ which verify
\begin{equation}
v''\leq 0\, ,    
\end{equation}
in the sense of distributions by \eqref{eq:halphaCD}, $v(0)=1$ and $v'(0)\leq H/(N-1)$ by \eqref{eq:halpha'Kdist}; and \eqref{eqn:LaplacianF}. Finally notice that $H\geq 0$, otherwise the bound 
\[
\Delta\mathrm{d}_{\overline{E}} \leq \frac{H}{1+\frac{H}{n-1}\mathrm{d}_{\overline E}}
\]
degenerates in finite time, contradicting the fact that $E$ is bounded, and the space is non-compact.
\end{proof}
}

\subsection{Isoperimetric inequality on non-negatively curved manifolds with Euclidean volume growth}

On a non-compact manifold $(M^n,g)$ with $\mathrm{Ric}\geq 0$ one can define, for $p\in M$,
\[
\mathrm{AVR}(M):=\lim_{r\to\infty}\frac{|B_r(p)|}{\omega_n r^n}.
\]
The definition does not depend on $p$ and it is well-posed by the Bishop--Gromov monotonicity formula \eqref{eqn:BishopGromov}. Notice that, by Bishop--Gromov monotonicity, $\mathrm{AVR}>0$ implies $\inf_{x\in M}|B_1(x)|\geq \omega_n\mathrm{AVR}>0$.

Manifolds with $\mathrm{Ric}\geq 0$ and $\mathrm{AVR}>0$ are asymptotically conical in the GH sense, see \cite{ChCo0}. More precisely, given $r_i\to \infty$, and $p\in M$, the rescaled metric space $(M,r_i^{-1}\mathrm{d},p)$ pG converges, up to subsequences, to a space $(C(X),\mathrm{d}_\infty,o)$ that is a \emph{metric cone} on a metric space with $\mathrm{diam}(X)\leq \pi$. Any such a space is called \emph{asymptotic cone} of $M$. 

Recall that, given a metric space $(X,\mathrm{d})$ with $\mathrm{diam}(X)\leq \pi$, the \emph{metric cone over $X$}, denoted by $C(X)$, is the space $[0,+\infty)\times X$, where $\{0\}\times X$ is collapsed to a point, endowed with the distance
\[
\mathrm{d}_C((a,x_1),(b,x_2)):=\sqrt{a^2+b^2-2ab\cos(\mathrm{d}(x_1,x_2))}.
\]

The asymptotic cone might depend on the specific sequence $r_i\to \infty$ when $\mathrm{Ric}\geq 0$, see Perelman \cite{PerelmanExampleCones}, or Colding--Naber \cite{ColdingNaber}. However, if $\mathrm{Sect}\geq 0$, the cone is independent on the sequence, see Ballman--Gromov--Schroeder \cite{BallmannGromovSchroeder85}. 

By Bishop--Gromov monotonicity formula $\mathrm{Ric}\geq 0$ implies $|B_r(p)|\leq \omega_nr^n$ for every $p\in M$ and $r>0$. Hence, manifolds with $\mathrm{Ric}\geq 0$, and $\mathrm{AVR}>0$ have maximal volume growth at infinity. On the other side of the spectrum, since $\mathrm{Ric}\geq 0$ always implies $|B_r(p)|\geq C(p)r$ (by a result of Calabi--Yau) for every $r>1$, we have manifolds such that $|B_r(p)|\sim r$ as $r\to \infty$. Those manifolds are, in a sense that can be made precise, \emph{asymptotically cylindrical}. In this note we will not give more details on the structure of the latter class of spaces, on which it is also interesting to study the isoperimetric problem. Such a study has been carried over in A--Po \cite{AP23}, and Zhu \cite{XingyuZhu}.

Now, we are ready to state a sharp isoperimetric inequality on manifolds with $\mathrm{Ric}\geq 0$ and $\mathrm{AVR}>0$. We postpone the proof to next lecture.

\begin{theorem}[Agostiniani--Fogagnolo--Mazzieri \cite{AgostinianiFogagnoloMazzieri}, Brendle \cite{BrendleFigo}, Fogagnolo--Mazzieri \cite{FogagnoloMazzieri}, Balogh--Krist\'aly \cite{BaloghKristaly}, A--Pa--Po--S \cite{ConcavitySharpAPPS2}, Cavalletti--Manini \cite{CavallettiManiniRigidityIsop}]\label{sharpisop}
Let $(M^n,g)$ be a smooth complete non-compact Riemannian manifold with $\mathrm{Ric}\geq 0$ and $\mathrm{AVR}>0$. Let $E$ be a set of finite perimeter, and with finite volume. Then 
\begin{equation}\label{eqn:SharpIsopIneq}
P(E)\geq n(\omega_n\mathrm{AVR}(M))^{1/n}|E|^{\frac{n-1}{n}}.
\end{equation}
Moreover, if equality holds, then $M$ is isometric to $\mathbb R^n$, and there is a ball $B_r(p)\subset \mathbb R^n$ such that $|E\Delta B_r(p)|=0$.
\end{theorem}

\begin{remark}[On \cref{sharpisop} in more general spaces]
The proof of \cref{sharpisop} given in A--Pa--Po--S works in the more general setting of Alexandrov spaces with $\mathrm{Sect}\geq 0$, or even $\mathrm{RCD}(0,n)$ spaces, see also A-Pa-Po-V \cite{APPV23}, and \cite{CavallettiManiniRigidityIsop}. In such a larger class, the rigidity is reached on metric cones over spaces $X$ with a lower bound $\geq 1$ on the curvature, understood in the appropriate weak sense. 
\end{remark}

\newpage 

\section{Fourth Lecture}
\subsection{Proof of the sharp isoperimetric inequality on $\mathrm{Ric}\geq 0$ and $\mathrm{AVR}>0$}
\begin{proof}[Proof of \cref{sharpisop}]
    We only prove the inequality. We fix $v>0$, and we want to show that for every set of finite perimeter $E$ with $|E|=v$, \eqref{eqn:SharpIsopIneq} holds. We discuss two cases.
    \smallskip

    \textbf{Case 1}.  There is an isoperimetric set $E$ with volume $v$. By \eqref{eq:enlargement-ratio} we have
    \[
        \frac{P(E_r)^{\frac{n}{n-1}}}{|E_r|} \leq 
        \frac{P(E)^{\frac{n}{n-1}}}{|E|},
    \]
    for every $r\geq 0$. Let $r_i\to \infty$, fix $o\in M$, and let us consider the sequence of rescaled manifolds $(M,\mathrm{d}_i:=r_i^{-1}\mathrm{d},o)$ which converges (up to subsequence) to a metric cone $(C,\mathrm{d}_\infty,c)$. Notice that 
    \[
    \chi_{E_{r_i}}\mathcal{H}^n_{\mathrm{d}_i} \rightharpoonup \chi_{B_1^{\mathrm{d}_\infty}(c)}\mathcal{H}^n_{\mathrm{d}_\infty},
    \]
    in duality with continuous bounded functions in a realization of the pmGH convergence. Thus, by lower-semicontinuity of the perimeter (see \cite{AmborsioBrueSemola19}), by the convergence of volume measures, and by the fact that $|B_r^{\mathrm{d}_\infty}(c)|=\omega_n r^n\mathrm{AVR}(M)$ for every $r\geq 0$, we have 
    \begin{equation}\label{eqn:SharpIsopInequality}
    n^{\frac{n}{n-1}}\left(\omega_n\mathrm{AVR}(M)\right)^{1/(n-1)}=\frac{P_{C}(B_1^{\mathrm{d}_\infty}(c))^{\frac{n}{n-1}}}{|B_1^{\mathrm{d}_\infty}(c)|}\leq \liminf_{i\to\infty}\frac{P(E_{r_i})^{\frac{n}{n-1}}}{|E_{r_i}|} \leq 
        \frac{P(E)^{\frac{n}{n-1}}}{|E|}.
    \end{equation}
    Hence, since $P(\tilde E)\geq P(E)$ for every set of finite perimeter $\tilde E$ with $|\tilde E|=|E|=v$, we get the sought inequality.
    \medskip

    \textbf{Case 2.} There is not an isoperimetric set $E\subset M$ with volume $v$. Hence by \cref{concentrationcompactness}, there is a non-collapsed Ricci limit space at infinity $(X,\mathrm{d},x)$ and an isoperimetric set $E\subset X$, with $|E|_X=v$, such that
   \[
   P_X(E)=I_X(v)=I_M(v).
   \]
   Then, arguing as in \eqref{eqn:SharpIsopInequality} on $X$, and using the latter equalities, we have that for every set of finite perimeter $\tilde E\subset M$ with $|\tilde E|_M=|E|_X=v$ we have 
   \[
   P_M(\tilde E)\geq I_M(v)=P_X(E)\geq n^{\frac{n}{n-1}}\left(\omega_n\mathrm{AVR}(X)\right)^{1/(n-1)}|E|_X,
   \]
   and the proof is concluded because $\mathrm{AVR}(X)\geq \mathrm{AVR}(M)$, see \cref{IXIM}.
\end{proof}

\begin{remark}(About the rigidity of \cref{sharpisop})
An idea to prove the rigidity case of \cref{sharpisop}: when equality holds, then $P(E_r)^{\frac{n}{n-1}}/|E_r|$ is constant in $r\in (0,\infty)$. Hence, by \eqref{eqn:ControlHt}, we get that $\mathrm{Ric}(\nu_r,\nu_r)=0$ and $\partial E_r$ is totally umbilic. Then $M$ is a metric cone outside $\partial E$. With additional effort (see \cite{ConcavitySharpAPPS2} for the precise proof) one can show that $M$ must be a metric cone, and $E$ is a ball centered at one of the tips. Since $M$ is smooth it must be $\mathbb R^n$, and $E$ must be a ball.
\end{remark}

{
Let us give the proof of the inequality in \cref{sharpisop} as it is given in \cite{BaloghKristaly}. The following proof works for arbitrary $\mathrm{CD}(0,n)$ spaces with $\mathrm{AVR}>0$.

\begin{proof}[Proof of \cref{sharpisop} \cite{BaloghKristaly}]
Let us fix a bounded Borel set $E$ with positive finite volume. Let $D:=\diam E$. If in \cref{BMonCD0n} we set $A_0:=E$, and $A_1:=B_r(x_0)$, for $x_0\in E$ and $r>0$, then we have, using the notation $A_t$ in there,
\[
A_t \subset E_{t(r+D)}:=\{x\in M^n: \mathrm{d}(x,E)\leq t(r+D)\}.
\]
Thus, recalling that $\mathcal{M}(E)$ is the Minkowski content of $E$, we have, for $r>0$ fixed, using \cref{BMonCD0n},
\[
\begin{aligned}
    \mathcal{M}(E)&:=\liminf_{\varepsilon\to  0} \frac{|E_\varepsilon|-|E|}{\varepsilon} = \liminf_{t\to  0} \frac{|E_{t(r+D)}|-|E|}{t(r+D)} \geq \liminf_{t\to  0} \frac{|A_t|-|E|}{t(r+D)}   \\
    &\geq \liminf_{t\to  0} \frac{\left((1-t)|E|^{1/n}+t|B_r(x_0)|^{1/n}\right)^n-|E|}{t(r+D)} = \frac{n\left(|B_r(x_0)|^{1/n}-|E|^{1/n}\right)|E|^{(n-1)/n}}{(r+D)}.
\end{aligned}
\]
If we now let $r\to \infty$ in the previous inequality we get
\[
\mathcal{M}(E)\geq n\left(\omega_n\mathrm{AVR}(M)\right)^{1/n}|E|^{\frac{n-1}{n}}.
\]
Finally, to pass from the Minkowski content to the perimeter in the above inequality (thus concluding the proof), we need to use the fact that the lower semicontinuous relaxation of $\mathcal{M}(\cdot)$ with respect to the convergence in measure (on finite measure sets) is the perimeter, see again \cite{ADMG17}.
\end{proof}
}

\begin{remark}[About other proofs of \eqref{eqn:SharpIsopIneq}]
    There are other proofs of \eqref{eqn:SharpIsopIneq} in the smooth setting. Agostiniani--Fogagnolo--Mazzieri \cite{AgostinianiFogagnoloMazzieri}, and Fogagnolo--Mazzieri \cite{FogagnoloMazzieri} use curvature flow techniques, and they deduce the isoperimetric inequality from the Willmore inequality using an argument due to Huisken. Brendle \cite{BrendleFigo} gives a proof using the ABP method, inspired by a former argument of Cabrè. All the previous proofs obtain rigidity when additional assumptions on the smoothness of $\partial E$ are assumed.

    In A--Pa--Po--S \cite{ConcavitySharpAPPS2}, we obtain the rigidity in the case $M$ is a smooth manifold without asking further assumptions on the regularity of $\partial E$, and we carry the proof in the non-smooth setting of $\mathrm{RCD}(0,n)$ spaces (see also \cite{APPV23}) along the lines of the proof of \cref{sharpisop} presented above. For the rigidity case in the setting of essentially non-branching $\mathrm{CD}(0,n)$ spaces, see \cite{CavallettiManiniRigidityIsop}; for a proof of the sharp isoperimetric inequality on $\mathrm{MCP}(0,n)$ spaces, see \cite{CM22}.
\end{remark}

For what concerns the behavior of the isoperimetric profile in the case of manifolds with linear volume growth, one can ask the following question. Observe that, by A--Po \cite{AP23}, \eqref{eqn:Equivalence} holds if one assumes $\mathrm{Sect}\geq 0$ instead of $\mathrm{Ric}\geq 0$. Moreover, by \cite{XingyuZhu}, $\Leftarrow$ always holds in the following question.
\begin{question}
    Let $(M^n,g)$ be a smooth complete non-compact Riemannian manifold with $\mathrm{Ric}\geq 0$ and $\inf_{x\in M} |B_1(x)|>0$. Fix $p\in M$. Understand whether the following equivalence holds:
    \begin{equation}\label{eqn:Equivalence}
    \sup_{v\in (0,\infty)}I_M(v)<\infty \Leftrightarrow \limsup_{r\to\infty}\frac{|B_r(p)|}{r}<\infty.
    \end{equation}
\end{question}

\subsection{Isoperimetric sets on non-negatively curved manifolds}
In this section we discuss how to improve \cref{concentrationcompactness} to existence of isoperimetric sets in some cases, when the ambient manifold has non-negative curvature.
\begin{theorem}[Ritoré \cite{RitoreExistenceSurfaces01}]\label{ExistenceSurface}
    Let $(M^2,g)$ be a smooth complete non-compact Riemannian surface with $\mathrm{Sect}\geq 0$. Then, for every $v>0$, there exists an isoperimetric set with volume $v$ in $M$.
\end{theorem}
\begin{proof}
    We sketch the proof following the alternative argument given in \cite{AP23}, where the result above is generalized to $2$-dimensional Alexandrov spaces. See \cref{ExCompleteAP}.
    
    First we can assume that $M$ has only one end, otherwise, by Cheeger--Gromoll splitting theorem, either $M\cong \mathbb R^2$ or $M\cong \mathbb R\times \mathbb S^1(r)$, where $\mathbb S^1(r)$ is a circle of radius $r>0$. In those cases isoperimetric sets exist for every volume.

    Then, one has\footnote{This is not immediate (see \cite{AP23}), and might not be true when $n\geq 3$, see the counterexamples in \cite{CrokeKarcher}.} $\inf_{x\in M}|B_1(x)|>0$. Let $v>0$ be a volume for which, by contradiction, there is no isoperimetric set. Hence, by \cref{concentrationcompactness}, there is a Ricci-limit space $(X,\mathrm{d})$ at infinity and an isoperimetric set $E\subset X$ with volume $v>0$ such that $P_X(E)=I_X(|E|)=I_M(v)$.

    Notice that one has\footnote{Here, $\mathrm{Sect}\geq 0$ is important, see \cite{AP23} It might be false when $n\geq 3$ and one only assumes $\mathrm{Ric}\geq 0$.} $X=\mathbb R\times Y$. Finally, we have two cases:
    \begin{enumerate}
        \item $Y\cong \mathbb R$. Picking a ball $B_{r_v}(p)\subset M$ such that $|B_{r_v}(p)|=v$, we have 
        \[
        I_{\mathbb R^2}(v)=I_M(v) \leq P(B_{r_v}(p)) \leq I_{\mathbb R^2}(v),
        \]
        where the last inequality comes from the Bishop--Gromov monotonicity formula \eqref{eqn:BishopGromov}. Hence, equality holds everywhere, and $B_{r_v}(p)$ is an isoperimetric set, reaching a contradiction.
        \item $Y\cong \mathbb S^1(r)$, for some $r>0$. Hence, if $v\ll r$ the isoperimetric set $E\subset X$ is a Euclidean ball and one concludes as in the previous item. Otherwise $E$ is a slab in $\mathbb R\times\mathbb S^1(r)$, and then $P_X(E)=2|\mathbb S^1(r)|$. Then, take any ray $\gamma$ in $M$ issuing from one point $o$. Consider $b_\gamma$ the Busemann function related to $\gamma$. One has, as $s\to\infty$,
        \[
        P(\{b_\gamma\leq s\})\nearrow |\mathbb S^1(r)|.
        \]
        Then there is $s$ for which $|\{b_\gamma\leq s\}|=v$, and 
        \[
        2|\mathbb S^1(r)|=I_{\mathbb R\times \mathbb S^1(r)}(v)=I_M(v)\leq P(\{b_\gamma\leq s\}) \leq |\mathbb S^1(r)|,
        \]
        resulting in a contradiction.\qedhere
    \end{enumerate}

\end{proof}

\begin{theorem}[A--B--F--Po \cite{AntonelliBrueFogagnoloPozzetta}]\label{ABFP}
    Let $(M^n,g)$ be a smooth complete non-compact Riemannian manifold with $\mathrm{Sect}\geq 0$ and $\mathrm{AVR}>0$. Then isoperimetric sets exist for every sufficiently large volume.
\end{theorem}
\begin{proof}
    We give a sketch, see \cref{DetailsABFP}. By Cheeger--Gromoll's splitting theorem we have $M\cong \mathbb R^k\times X$, with $0\leq k\leq n$, and $X$ does not split a line. By arguing on $X$ we can assume without loss of generality $k=0$. 
    
    Since $M$ does not split a line, and $\mathrm{Sect}\geq 0$, then the (unique) asymptotic cone of $M$ does not split any line, see \cite{BallmannGromovSchroeder85}. Hence (see \cite{AntonelliBrueFogagnoloPozzetta} for details) there exists $\varepsilon>0$ such that the following holds: for any limit space at infinity $(X,\mathrm{d}_X)$ of $M$ we have
    \begin{equation}\label{eqn:AVRinequality}
        \mathrm{AVR}(X)\geq \mathrm{AVR}(M)+\varepsilon.
    \end{equation}
    Now, observe that by using the sharp isoperimetric inequality (\cref{sharpisop}) and Bishop--Gromov monotonicity formula \eqref{eqn:BishopGromov} one can prove that 
    \begin{equation}\label{eqn:AsymptoticsI}
    \lim_{v\to \infty}\frac{I_M(v)}{n(\omega_n\mathrm{AVR}(M))^{\frac{1}{n}}v^{\frac{n-1}{n}}}=1,    
    \end{equation}
    and thus there exists $v_0$ such that for every $v\geq v_0$ we have 
\begin{equation}\label{eqn:LimitI}
    I_M(v)\leq n(\omega_n(\mathrm{AVR}(M)+\varepsilon/2))^{1/n}v^{\frac{n-1}{n}}.
    \end{equation}

    We are now ready to prove that there exist isoperimetric sets for every volume $v\geq v_0$. If not, 
    by \cref{concentrationcompactness}, there is a non-collapsed Ricci limit space at infinity $(X,\mathrm{d}_X,x)$ and an isoperimetric set $E\subset X$, with $|E|_X=v$, such that
   \begin{equation}\label{eqn:UsualEquality}
   P_X(E)=I_X(v)=I_M(v).
   \end{equation}
    Putting together \eqref{eqn:LimitI}, \eqref{eqn:UsualEquality}, \cref{sharpisop}\footnote{Observe: the proof of \cref{sharpisop} works also for non-collapsed Ricci-limit spaces because \cref{concentrationcompactness} holds in that setting too.}, and \eqref{eqn:AVRinequality}, we thus have 
    \begin{equation}
    \begin{aligned}
    n(\omega_n(\mathrm{AVR}(M)+\varepsilon/2))^{1/n}v^{\frac{n-1}{n}} &\geq I_M(v) = P_X(E)\geq n(\omega_n\mathrm{AVR}(X))^{1/n}v^{\frac{n-1}{n}} \\ &\geq n(\omega_n(\mathrm{AVR}(M)+\varepsilon))^{1/n}v^{\frac{n-1}{n}},
    \end{aligned}
    \end{equation}
    which is a contradiction, as desired.
\end{proof}

\begin{theorem}[A--G \cite{AntonelliGlaudo}]
    For every $n\geq 3$ there exists a smooth complete Riemannian manifold $(M^n,g)$ with $\mathrm{Sect}\geq 0$, and $\mathrm{AVR}>0$, such that: there is \underline{no} isoperimetric set with volume $v< 1$.
\end{theorem}
\begin{proof}
    We give a sketch and a picture. Let us describe the idea of the construction. The main idea is to first obtain the same non-existence result for the (relative) isoperimetric problem in unbounded convex bodies. Then, at the same time, with extra care, one shows that the  non-existence result is enjoyed by the boundary of such convex bodies with the intrinsic distance. Finally, one smoothens the boundary.

The construction goes as follows: pick a convex cone $\Sigma\subseteq\mathbb R^n$ contained in the half-space $\{x\in\mathbb R^n: x\cdot e_1\geq 0\}$, where $e_1$ is the first vector of the canonical basis of $\mathbb R^n$. For a convex decreasing function $\varphi:\mathbb R\to(0,\infty)$ with $\varphi(+\infty)=0$, we define $C$ as 
\begin{equation*}
    C:= (\Sigma\times\mathbb R)\cap \{(x, t)\in\mathbb R^n\times\mathbb R:\, x\cdot e_1\geq \varphi(t)\}.
\end{equation*}
A rather delicate choice of $\varphi$ (so that $\varphi'$ is pointwise small compared to $\varphi$) produces a convex set $C$ without isoperimetric sets with small volume. The idea is that one limit at infinity of $C$ is the cone $\Sigma\times\mathbb R$, which has density strictly lower than any tangent cone at a point of $C$. This already implies that if there were isoperimetric sets for small volumes $v$ in $C$, then these sets would escape to infinity as $v\to 0$, by \cite{LeonardiRitore}.

The crux of the argument is to use an a-priori estimate on the diameter of isoperimetric sets to upgrade this observation to the non-existence of isoperimetric sets for small volumes.

\begin{figure}[htbp]\label{fig}

 
\begin{tikzpicture}[scale=1.3,
declare function={
phi(\x)= (\x < 0) * (1/1.412-0.5*\x)   +
         (\x >= 0) * (1/(\x+1.412))
;
}]
 
\begin{axis} [
    axis lines = center, 
    axis on top, 
    axis line style = {very thin, black!20},
    view = {20}{20},
    unit vector ratio = 1 1 1,
    xtick = {0},
    ytick = {0},
    ztick = {0},
    xlabel = {{\tiny $t$}},
    zlabel = {{\tiny $x_1$}},
    x label style={anchor=north},
]
 
\addplot3 [
    domain=-0.7:5,
    domain y = -2:2,
    samples = 20,
    samples y = 100,
    surf,
    shader = interp,
    opacity = 0.4,
    colormap/gray] {max(phi(x), abs(y))};
    
\addplot3+[
    black!60,
    no markers,
    samples=51, 
    samples y=0,
    domain=-0.7:5,
    variable=\x]
    ({\x},{phi(\x)},{phi(\x)});

\addplot3+[
    black!60,
    no markers,
    samples=51, 
    samples y=0,
    domain=-0.7:5,
    variable=\x]
    ({\x},{-phi(\x)},{0.01+phi(\x)});

\end{axis}
 
\end{tikzpicture}
 
    \caption{Representation of the boundary $\partial C$.}
\end{figure}
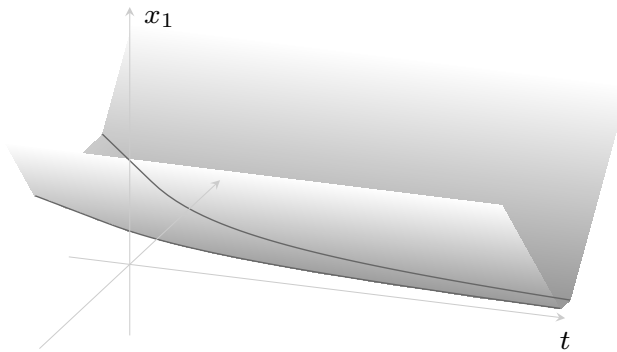
\end{proof}

We finish with few more open questions.

\begin{question}
    Understand whether there exists a smooth complete Riemannian manifold $(M^n,g)$ with $\mathrm{Ric}\geq 0$, $\mathrm{AVR}>0$, and \underline{no} isoperimetric sets.
\end{question}

\begin{question}
    Understand whether the following is true. Let $(M^n,g)$ be a smooth complete Riemannian manifold with $\mathrm{Sect}\geq 0$ {and $\inf_{x\in M}|B_1(x)|>0$}. Then isoperimetric sets exist for every sufficiently large volume.
\end{question}

\begin{question}\label{questionconvexity}
    Does there exist a non-compact smooth complete Riemannian manifold $(M^3,g)$ with $\mathrm{Sect}\geq 0$, { $\inf_{x\in M}|B_1(x)|>0$}, and with an isoperimetric set which is \underline{not} convex?
\end{question}
\newpage

\section{Exercises}
\begin{exercise}\label{ExNonExistence}
    Construct a Riemannian manifold $M$ such that:
    \begin{itemize}
        \item There are $v_1,v_2\in (0,|M|)$ such that: there exists an isoperimetric set with volume $v_1$ \textbf{and} there is \underline{no} isoperimetric set with volume $v_2$.
        \item For every volume $v\in (0,|M|)$ there is no isoperimetric set with volume $v$.
    \end{itemize}
\end{exercise}
\begin{exercise}[Maz'ya, Federer--Fleming -- Sobolev inequality]
Show that the isoperimetric inequality in \eqref{eqn:Isoperimetric} is equivalent to the $(1,n/(n-1))$-Sobolev inequality: 
\[
\int_{\mathbb R^n}|\nabla f| \geq n\omega_n^{1/n}\left(\int_{\mathbb R^n}|f|^{n/(n-1)}\right)^{(n-1)/n}, \qquad \forall f\in C^\infty_c(\mathbb R^n).
\]
\emph{Hint}. Use \cref{Approx}, \cref{Coarea} in \cref{subsectPropAndComm}, and (the properly rescaled version of) \eqref{eqn:Isoperimetric}.
\end{exercise}
\begin{exercise}\label{BMinequality}
    Prove the \emph{Brunn--Minkowski inequality}. Let $A,B$ be two compact sets in $\mathbb R^n$. Then 
    \[
    |A+B|^{1/n} \geq |A|^{1/n}+|B|^{1/n}.
    \]

    \emph{Hint}. First, prove it for pluri-rectangles $A=\prod_{i=1}^n I_i$, $B=\prod_{i=1}^n J_i$, where $\{I_i,J_i\}$ are compact intervals in $\mathbb R$. Then, prove it for $A,B$ being disjoint unions of pluri-rectangles. Finally, deduce the general version by approximation. This proof can be found in Federer's book \cite{Federer}.
\end{exercise}

\begin{exercise}
    Use the Brunn--Minkowski inequality to prove the following. If $K\subset\mathbb R^n$ is a convex set, and $K_t:=\{x\in \mathbb R^n:\mathrm{d}(x,K)\leq t\}$, for $t\geq 0$, then 
    \[
    [0,\infty)\ni t\mapsto |K_t|^{1/n}, \qquad \text{is concave}.
    \]
    Compare this with (the last part of) \cref{MeanCurvatureMonotonicity}, and \cref{questionconvexity}.
\end{exercise}

\begin{exercise}\label{ExRegularity}
  Following the outline in \cref{rem:RegularityIM}, show that under the hypotheses of \cref{thm:ConcavitySharp} the isoperimetric profile $I_M$ is locally semi-concave.  
\end{exercise}

\begin{exercise}
    Write down explicitly all the computations outlined in \cref{rem:Equality}.
\end{exercise}

\begin{exercise}\label{ExODEDetails}
  Write down explicitly all the details about the ODE comparison argument in \cref{LGBG} (The interested reader can find the complete argument in the appendix of \cite{PozzettaSurvey}).
\end{exercise}

\begin{exercise}\label{RicLocDoub}
    Prove that a Riemannian manifold $(M^n,g)$ with $\mathrm{Ric}\geq k\cdot g$, for some $k\in\mathbb R$, is locally doubling. Moreover, in such a case, given $R>0$, the constant $C(R)$ in \eqref{eqn:LocallyDoubling} only depends on $R,k,n$.

    \emph{Hint}. Use \eqref{eqn:BishopGromov}.
\end{exercise}

\begin{exercise}\label{exTrivialI}
    Let $(M^n,g)$ be a smooth complete non-compact Riemannian manifold with $\mathrm{Ric}\geq 0$. Then 
    \[
    I_M\equiv 0 \Leftrightarrow \inf_{x\in M}|B_1(x)|=0.
    \]

    \emph{Hint}. Use Bishop--Gromov monotonicity to produce balls with large volume and small perimeters.
\end{exercise}

\begin{exercise}\label{ExComputationsAG}
    Work out the detailed computations in the last nine lines of the proof of \cref{MeanCurvatureMonotonicity}.
\end{exercise}

\begin{exercise}\label{IXIM}
    Let $(M,g)$ be a smooth complete non-compact Riemannian manifold without boundary with $\mathrm{Ric}\geq k\cdot g$, and $\inf_{x\in M}|B_1(x)|>0$. Let $x_i$ be a sequence of diverging points on $M$ such that 
    \[
    (M,\mathrm{d},x_i)\longrightarrow (X,\mathrm{d}_X,x)
    \]
    in the pGH sense. 
    \begin{itemize}
        \item Prove that $I_X\geq I_M$.
        \item Assume that $k=0$. Prove that $\mathrm{AVR}(X)\geq \mathrm{AVR}(M)$.
    \end{itemize}
\end{exercise}

\begin{exercise}\label{ExCompleteAP}
    Add as many details as possible to the proof sketched above in \cref{ExistenceSurface}.
\end{exercise}

\begin{exercise}\label{DetailsABFP}
     Add as many details as possible to the proof sketched above in \cref{ABFP}. In particular prove \eqref{eqn:AVRinequality}, and \eqref{eqn:AsymptoticsI}.
\end{exercise}
\medskip

\textbf{Acknowledgments}. The material is based on my collaborations and papers with Elia Bruè (B), Mattia Fogagnolo (F), Federico Glaudo (G), Stefano Nardulli (N), Enrico Pasqualetto (Pa), Marco Pozzetta (Po), Daniele Semola (S), Ivan Violo (V), Kai Xu (X). I thank them for the many fruitful conversations and collaborations about the isoperimetric problem in the last few years. I am also deeply grateful to the organizers and the people who attended the \emph{European Doctorate School of Differential Geometry}. I would like to thank Žan Bajuk, Marco Pozzetta, Daniele Semola, {and the anonymous reviewers} for having read a first draft of these Lecture Notes. I am also indebted to the organizers and the people who attended the summer school \emph{Optimal transport, heat flow and synthetic Ricci bounds} in honor of Luigi Ambrosio, 2024 Nemmers Prize recipient.  
\medskip

\textbf{Funding Information}.
I have been partially supported by the AMS-Simons Travel Grant, and the NSF DMS Grant No. 2550590.
\smallskip

\textbf{Author contribution}.
The author confirms the sole responsibility for the conception of the study, presented results and manuscript preparation.
\medskip

\textbf{Conflict of interest}. 
The author states no conflict of interest.
\medskip

\printbibliography

\end{document}